\documentclass[11pt]{article}
\usepackage[a4paper,margin=1in]{geometry}
\usepackage{amsmath,amssymb,amsthm,mathtools}
\usepackage{bm}

\usepackage{xcolor}
\usepackage{subcaption}
\usepackage{algorithm}
\usepackage{algpseudocode}
\usepackage{comment}

\usepackage{graphicx}
\usepackage{enumitem}

\usepackage[capitalise]{cleveref}
\usepackage{aliascnt}

\newtheorem{theorem}{Theorem}
\newaliascnt{proposition}{theorem}
\newtheorem{proposition}[proposition]{Proposition}
\aliascntresetthe{proposition}
\newaliascnt{remark}{theorem}
\newtheorem{remark}[remark]{Remark}
\aliascntresetthe{remark}
\newaliascnt{lemma}{theorem}
\newtheorem{lemma}[lemma]{Lemma}
\aliascntresetthe{lemma}
\newtheorem{assumption}{Assumption}
\newtheorem{corollary}[theorem]{Corollary}

\crefname{proposition}{Proposition}{Propositions}
\crefname{remark}{Remark}{Remarks}
\crefname{lemma}{Lemma}{Lemmas}
\crefname{assumption}{Assumption}{Assumptions}

\newcommand{\R}{\mathbb{R}}

\newcommand{\Hcal}{\mathcal{H}}
\newcommand{\Acal}{\mathcal{A}}
\newcommand{\supp}{\operatorname{supp}}

\newcommand{\ri}{\operatorname{ri}}

\newcommand{\MT}{\operatorname{MT}}
\newcommand{\argmin}{\operatorname{argmin}}

\newcommand{\dd}{\,\mathrm{d}}
\newcommand{\Id}{\mathrm{Id}}

\newcommand{\eps}{\varepsilon}

\newcommand{\ph}{\varphi}

\title{On the Linear Convergence of SISTA for Entropic \\
Martingale Optimal Transport}
\author{Giorgia Bifronte, Armand Ley, Maxime Sylvestre}
\date{}

\begin{document}
\maketitle

\begin{abstract}
We prove a linear convergence rate of a Sinkhorn-type algorithm for entropic martingale optimal transport in arbitrary dimension. The algorithm follows a block-coordinate ascent scheme, alternating exact Sinkhorn updates for the marginal potentials with an explicit gradient step for the potential associated with the martingale constraint. Under irreducibility and non-degeneracy assumptions on the compactly supported marginals, we establish linear convergence of the dual objective values and iterates. Our analysis relies on a pointwise coercivity estimate involving the gradient of the dual functional with respect to the potential for the martingale constraint. This property allows us to establish uniform $L^{\infty}$ bounds on the iterates. Combining these bounds with smoothness and strong convexity modulo affine invariances yields a uniform Polyak--{\L}ojasiewicz inequality along the iterations and, consequently, a linear convergence rate. We complement our theoretical results with numerical experiments illustrating the convergence behavior of the algorithm. 
\end{abstract}

\section{Introduction}

In recent years, martingale optimal transport has attracted considerable attention, partly due to its applications in mathematical finance \cite{Ho98,BeHePe13,GaHeTo14,DoSo14,Ho11,He17} and, more recently, generative modelling \cite{AlBaCaPh25,AlHeLoMaPhTo26,DePhZa26}. In his seminal article, Strassen \cite{Stra65} (see also \cite{CaFeMe64}) 
proved that there exists a martingale coupling between $\mu$ and $\nu$ if and only if $\mu$ is smaller than $\nu$ in the convex order. Since its proof is non-constructive and martingale couplings are generally not unique, a substantial literature has sought principles to exhibit distinguished martingale transports, for instance  through solutions of the Skorokhod embedding problem \cite{Ob04,Ho11,BeCoHu17} or, most often, as solutions of variational problems. Such distinguished couplings can be obtained by solving static linear \cite{BeJu16,HeTo16,BeHeTo17,GhKiLi19} or nonlinear transport problems \cite{GoRoSaTe17,BaBePa19,BeJu21,BaBeHuKa20,BaBeSchaTschi26,BaLoOb25,NuWi24,BaBeBiLe26}, or by considering the endpoint laws of martingales optimizing dynamic variational problems with prescribed initial and terminal distributions \cite{BaBeHuKa20,HuTr17,BaBeBiLe26}. This raises the computational question of how to construct martingale transport plans efficiently.

We briefly recall the main computational approaches to martingale optimal transport. If both marginals are finitely supported, the martingale optimal transport problem can be formulated as a linear program and solved using standard linear programming methods \cite{GuOb19,AlCoJo19,EcGuLiOb21}. PDE methods solve the dual formulation of continuous-time martingale transport, which involves a fully nonlinear Hamilton--Jacobi--Bellman equation, either with prescribed marginals \cite{BoTa13,TaTo13} or with finitely many expectation constraints
\cite{GuLo21,GuLoWa22}. Dedicated algorithms have also been developed for specific martingale transport problems, notably the Bass martingale problem \cite{CoHe21,AcMaPa25,HJLOP,BaPaSc25}.
Another standard approach, which we follow in this paper, is to add entropic regularization enabling Sinkhorn-type algorithms: we discuss this approach and review the related literature in more detail below.

\paragraph{Entropic Martingale Optimal Transport}

We consider the entropically regularized martingale optimal transport problem with cost function $c$ and regularization parameter $\varepsilon>0$:
\begin{align}\label{eq:introEMOT} \tag{\text{EMOT$_{\varepsilon}$}}
		\inf_{\substack{m\in \MT(\mu,\nu)}} \int c(x,y) \dd m(x,y) + \varepsilon H(m\mid \mu\otimes \nu),
	\end{align}
where $$\MT(\mu,\nu) \coloneqq \{ m \in \Pi(\mu,\nu): \int_{\R^d} y \dd m_x(y)=x \quad\mu \text{-a.e. } x \}$$ denotes the set of martingale couplings, with $\Pi(\mu,\nu)$ being the set of couplings between $\mu$ and $\nu$, and $H(m\mid \mu\otimes \nu)$ denotes the relative entropy, i.e.\ equals $\int \log\frac{\dd m}{\dd \mu \otimes \nu}\dd m$ if $m\ll \mu \otimes \nu$ and $+\infty$ otherwise.
This problem is of interest for two main reasons.

First, as $\varepsilon\to 0$, it provides an approximation
of the unregularized martingale optimal transport problem.
While this convergence is standard in the discrete setting,
it also holds for general marginals satisfying irreducibility
and non-degeneracy assumptions: under suitable regularity
assumptions on $c$, the entropically regularized functionals
$\Gamma$-converge to the unregularized functional.
Moreover, quantitative convergence rates for the optimal
values are available under additional regularity assumptions
on the cost~\cite{CaMaSy26}.

Second, entropic martingale couplings are of independent
interest. In particular, when $c=0$, the problem reduces
to minimizing relative entropy over martingale couplings.
Its optimizer is the (static) martingale Schrödinger bridge,
introduced in \cite{NuWi24} and subsequently extended and
characterized in \cite{BaBeBiLe26}.
This coupling can be interpreted as the martingale coupling
closest to independence in the sense of relative entropy.
It has found applications in the study of the football
model \cite{GuJuTa26} and in the resolution of a conjecture
of Talagrand \cite{HuSoTu26}.
Furthermore, \cite{BaBeBiLe26} establishes a connection
with a continuous-time calibration problem in mathematical
finance. Under suitable assumptions on the marginals,
the martingale Schrödinger bridge determines a
continuous-time martingale, known as the Föllmer martingale,
which solves the variational problem
\begin{align}\label{eq:introP_cont}\tag{$P^{cont}$}
		\inf_{\substack{M_0\sim\mu,\;M_1\sim\nu,\, M_t=M_0+\int_0^t \sigma_s\dd B_s}}
		\frac12\,\mathbb{E}\Big[\int_0^1 \frac{|\sigma_t-I_d|_{\text{HS}}^2}{1-t}\dd t\Big].
\end{align}

As in classical optimal transport, \labelcref{eq:introEMOT} admits a dual formulation. For a triple of potentials $(\varphi,\psi,h)$, with $\varphi,\psi\colon\R^d\to\R$ and $h\colon\R^d\to\R^d$, define $\Phi(x,y)\coloneqq\varphi(x)+\psi(y)+h(x)\cdot(y-x)$. The corresponding dual functional is given by 
\begin{equation}\label{eq:introDual}
\mathcal{D}(\varphi,\psi,h)\coloneqq\int\varphi\,\dd\mu+\int\psi\,\dd\nu-\varepsilon\int\exp\Big(\frac{\Phi(x,y)-c(x,y)}{\varepsilon}\Big)\dd\mu\otimes \nu(x,y)+\varepsilon .
\end{equation}
For $m\in\MT(\mu,\nu)$ with finite entropy and integrable potentials, writing $\rho\coloneqq\frac{\dd m}{\dd(\mu\otimes\nu)}$ and $u\coloneqq(\Phi-c)/\varepsilon$, the Fenchel--Young inequality $u\rho\le\rho\log\rho-\rho+e^{u}$ then yields
\begin{equation} \label{eq:intro_weak_duality}
    \int c\,\dd m+\varepsilon H(m\mid\mu\otimes\nu)-\mathcal D(\varphi,\psi,h)=\varepsilon\int\big(\rho\log\rho-\rho-u\rho+e^{u}\big)\,\dd(\mu\otimes\nu)\ge0,
\end{equation}
with equality if and only if $\rho=e^{u}$. 
This establishes weak duality and it shows that a dual maximizer, when it exists, determines the primal optimizer. Moreover, \labelcref{eq:introEMOT} can be written as  
$$
\inf_{m\in\Pi(\mu,\nu)}\int C(x,m_x)\dd \mu(x), 
\quad \text{with} \; \;
C(x,p)\coloneqq\int c(x,y)\dd p(y)+\varepsilon H(p\mid\nu)+\infty \cdot \mathbf{1}_{\int y\dd p(y)\neq x}.$$ 
Thus, \labelcref{eq:introEMOT} is a weak optimal transport problem in the sense of \cite{GoRoSaTe17}, with the martingale condition imposed as a hard moment constraint, as studied in \cite{CaMaSy26}. For further developments on weak transport duality, see \cite{GoRoSaTe17,GoRoSaShTe18,AlBoCh19,BaBePa19,BaPa22,BePaRiSc25}.

While strong duality holds under mild assumptions, dual attainment is more delicate in the presence of moment constraints. It was established in dimension one in \cite{NuWi24} and, for compactly supported marginals satisfying a suitable qualification condition, implied by irreducibility and non-degeneracy, in \cite{CaMaSy26}.

Whenever $\mathcal D$ admits a maximizer $(\varphi,\psi,h)$, the equality case in \cref{eq:intro_weak_duality} implies that the optimizer of \labelcref{eq:introEMOT} has the Gibbs density
\begin{equation}\label{eq:introGibbs}
	\frac{\dd m_\varepsilon}{\dd(\mu\otimes\nu)}(x,y)=\exp\Big(\frac{\varphi(x)+\psi(y)+h(x)\cdot(y-x)-c(x,y)}{\varepsilon}\Big).
\end{equation}
The first-order optimality conditions with respect to $\varphi$, $\psi$ and $h$ recover the two marginal constraints and the martingale constraint:
\[
\int e^{u(x,y)}\dd\nu( y)=1\ \ \mu\text{-a.e.},\quad 
	\int e^{u(x,y)}\dd\mu( x)=1\ \ \nu\text{-a.e.},
\]
\begin{equation}
    \label{eq:introFOC}
    \text{and}\quad
    \int(y-x)\,e^{u(x,y)}\dd\nu( y)=0\ \ \mu\text{-a.e.}
\end{equation}
In the following we minimize the convex functional $F\coloneqq\varepsilon-\mathcal D$ rather than maximize $\mathcal D$. 

\paragraph{SISTA and Main Result}

A standard approach to entropic optimal transport is Sinkhorn's algorithm \cite{Cu13}, which alternates exact minimization steps in the dual variables and converges linearly. In the martingale setting, the dual formulation involves an additional multiplier $h$ associated with the martingale constraint. Existing Sinkhorn-type algorithms typically update $h$ by exact minimization \cite{CheCoReWa26,KaReWa26}. Unlike the marginal updates, however, this step is implicit and requires solving, for each $x$, a $d$-dimensional nonlinear convex problem characterized by \cref{eq:introFOC}. To avoid this computational burden, we follow \cite{CaMaSy26} and employ the SISTA (Sinkhorn--Iterative--Shrinkage--Thresholding--Algorithm) introduced in \cite{CaDuGaSu23}, replacing the exact minimization in $h$ with an explicit gradient step.\footnote{The original SISTA employs a proximal gradient step to promote sparsity. In the absence of a sparsity penalty, as in our setting, this reduces to a standard gradient step.}

Starting from an initial triple $(\ph^0,\psi^0,h^0)$
and a step size $\tau>0$, the iterates are defined recursively by
\begin{equation*}
\begin{aligned}
\psi^{n+1}
&= \argmin_{\psi} F(\ph^n,\psi,h^n),\\
\ph^{n+1}
&= \argmin_{\ph} F(\ph,\psi^{n+1},h^n),\\
h^{n+1}
&= h^n-\tau\nabla_h F(\ph^{n+1},\psi^{n+1},h^n),
\end{aligned}
\end{equation*}
which can be rewritten in terms of explicit formulas as
\begin{equation}\label{eq:introSISTA}
\begin{aligned}
	\psi^{n+1}(y)&=-\varepsilon\log\int\exp\Big(\frac{\varphi^{n}(x)+h^{n}(x)\cdot(y-x)-c(x,y)}{\varepsilon}\Big)\dd\mu( x),\\
	\varphi^{n+1}(x)&=-\varepsilon\log\int\exp\Big(\frac{\psi^{n+1}(y)+h^{n}(x)\cdot(y-x)-c(x,y)}{\varepsilon}\Big)\dd\nu( y),\\
	h^{n+1}(x)&=h^{n}(x)-\tau\int(y-x)\dd\pi^{n}_x( y),
\end{aligned}
\end{equation}
where $$\dd \pi^{n}_x(y)\coloneqq\exp\left(\frac{\varphi^{n+1}(x)+\psi^{n+1}(y)+h^{n}(x)\cdot(y-x)-c(x,y)}{\varepsilon}\right)\dd \nu(y).$$ 

Our main result, stated in the next section, establishes linear convergence of both the objective values and the iterates of SISTA under compactness, non-degeneracy, and irreducibility assumptions on the marginals. 

\paragraph{Difficulties and Strategy of the Proof}

Block optimization methods are known to converge linearly if the underlying map $F$ is strongly convex and has Lipschitz gradients \cite{BeLu13}. However, the dual functional $F$ associated to our entropic martingale optimal transport problem introduces three difficulties: 
\begin{enumerate}
    \item $F$ is neither globally smooth nor globally strongly convex. It is finite and differentiable on $\mathcal H^\infty := L^\infty(\mu)\times L^\infty(\nu)\times L^\infty(\mu;\R^d)$, but its smoothness and convexity constants deteriorate exponentially with $\|(x,y) \mapsto [\varphi(x)+\psi(y)+h(x)\cdot(y-x)-c(x,y)]\|_{L^\infty(\mu\otimes\nu)}$;
    \item The iterates are not a priori bounded in $L^\infty(\mu)\times L^\infty(\nu)\times L^\infty(\mu;\R^d)$, so it is not clear if they remain in a region where these constants can be controlled;
    \item The function $F$ is invariant along the affine gauge $$(\varphi,\psi,h)
\longmapsto
\bigl(\varphi-a\cdot x+b,\,\psi+a\cdot y-b,\,h-a\bigr),
\qquad
a\in\R^d,\ b\in\R,$$ thus, it is not strongly convex and its minimizers are not unique.
\end{enumerate}

Our strategy of proof is organized accordingly. In \cref{subsection:abstract}, we establish an abstract convergence result for cyclic block-coordinate methods on Hilbert spaces combining exact block minimization and gradient steps. Assuming that all  partial iterates remain in a closed convex set $K$ on which $F$ is convex, has a Lipschitz continuous gradient, is strongly convex in each exactly minimized block, and satisfies a Polyak--{\L}ojasiewicz inequality, we prove that the sub-optimality gap as well as the iterates converge linearly  with explicit rates. This result is an abstract version of the arguments of \cite[Thm.~3.1]{CaDuGaSu23} and \cite[Prop.~6.5]{GaNaTo25}, and may be of independent interest.

The main challenge is to establish uniform boundedness of the iterates. To this end, we exploit a weak coercivity property of $F$ to obtain uniform bounds on suitably normalized iterates. Using the gauge invariance of $F$, we then show that all partial iterates belong to a set of the form $K=\{ \xi=(\varphi,\psi,h)\,:\,\|(x,y) \mapsto [\varphi(x)+\psi(y)+h(x)\cdot(y-x)-c(x,y)]\|_{L^\infty(\mu\otimes\nu)}\le M\}$ on which the abstract convergence theorem applies.

\paragraph{Comparison with the Literature}
We compare our results with existing Sinkhorn-type algorithms for entropic martingale optimal transport. A common feature of these methods is that the marginal potentials $\varphi$ and $\psi$ admit explicit Sinkhorn updates, whereas the martingale multiplier $h$ is updated by exact minimization. For each $x$, this requires solving the nonlinear equation
\begin{equation*}
\int (y-x)\exp\bigl(\varphi(x)+\psi(y)+h^{n+1}(x)\cdot(y-x)-c(x,y)\bigr)\dd\nu(y)=0,
\end{equation*}
equivalently, a $d$-dimensional strictly convex minimization problem. This inner problem has no closed-form solution and needs to be solved numerically. 

The first Sinkhorn-type algorithm for martingale optimal transport was introduced by De March \cite{De18} for finitely supported marginals, using Newton's method to update $h$. Under the existence of a full-support martingale coupling, linear convergence of the objective values is established. In dimension one, Chen, Conforti, Ren and Wang \cite{CheCoReWa26} prove linear convergence for compactly supported marginals, assuming the existence of a martingale coupling with bounded log-density relative to $\mu\otimes\nu$. Eckstein and Lakhal \cite{EcLa26} establish linear convergence for general iterative proportional fitting procedures, with rates expressed in terms of Friedrichs angles between constraint spaces. Their analysis requires uniform bounds on the log-densities of the iterates and, in the martingale setting, a reference martingale coupling with bounded log-density relative to $\mu\otimes\nu$. More recently, Kazeykina, Ren and Wang \cite{KaReWa26} prove exponential convergence in relative entropy in $\R^d$, without assuming the existence of such a coupling.  Their result requires compactly supported marginals in strict convex order, a cost that is Lipschitz in each variable, and a convex support of $\nu$ containing $\supp(\mu)$ in its relative interior.

To the best of our knowledge, our result is the first to establish a convergence rate for a Sinkhorn-type algorithm for entropic martingale optimal transport without requiring an exact update of the martingale multiplier at each iteration. Unlike \cite{De18}, our analysis applies to general marginals, not necessarily finitely supported. Compared with \cite{CheCoReWa26,EcLa26}, it holds in arbitrary dimension without requiring a martingale coupling with bounded log-density. Moreover, unlike \cite{KaReWa26}, we do not assume that $\supp(\nu)$ is convex. This allows, in particular, finitely supported terminal marginals, such as the one illustrated in \cref{fig:2d_follmer}, which are especially relevant for numerical implementations. Under non-degeneracy, the strict convex order condition imposed in \cite{KaReWa26} is equivalent to our irreducibility assumption (see the Appendix).

Algorithms combining exact Sinkhorn updates with explicit gradient steps have also been studied in other settings. Gallouët, Natale and Todeschi \cite[Prop.~6.5]{GaNaTo25} establish linear convergence for entropically regularized barycentric costs in Wasserstein metric extrapolation with finitely supported marginals. Kato and Wang \cite{KaWa26} recently prove convergence of a similar scheme for entropic vector quantile regression with general compactly supported marginals.

Finally, the martingale Schrödinger bridge determines a canonical continuous-time martingale, the Föllmer martingale \cite{BaBeBiLe26}, which is the entropic counterpart of the Bass martingale \cite{BaBeHuKa20,BaBeSchaTschi26}. Both admit explicit representations in terms of their potentials and can be simulated using Gaussian convolutions. The main computational difference lies in the determination of these potentials. In dimension one, the Bass potential can be computed efficiently through the fixed-point iteration of \cite{CoHe21}, whose convergence is established in \cite{AcMaPa25,HJLOP}. In higher dimensions, however, each iteration of the martingale Sinkhorn algorithm \cite{HJLOP} involves solving an optimal transport problem, performing a Gaussian convolution and computing a Legendre transform, whereas SISTA requires only explicit updates.

\paragraph{Numerical Experiments}

After the theoretical analysis, we provide some numerics to illustrate convergence behavior. In our first example, we compute the solution of \labelcref{eq:emot-setting} with cost $c(x,y) = |y-x|$ and regularization parameter $\varepsilon$ going to $0^+$. We then compute the martingale Schrödinger bridge ($c=0$) for two-dimensional marginals and apply it to compute the associated Föllmer martingale (see \Cref{fig:2d_follmer}). Finally, we numerically illustrate the role of the irreducibility and non-degeneracy assumptions: loss of irreducibility is accompanied by loss of multiplier boundedness, while approaching a degenerate regime leads to a progressive deterioration of the observed convergence rate.
\begin{figure}[H]
    \centering
    
    \includegraphics[width=\textwidth]{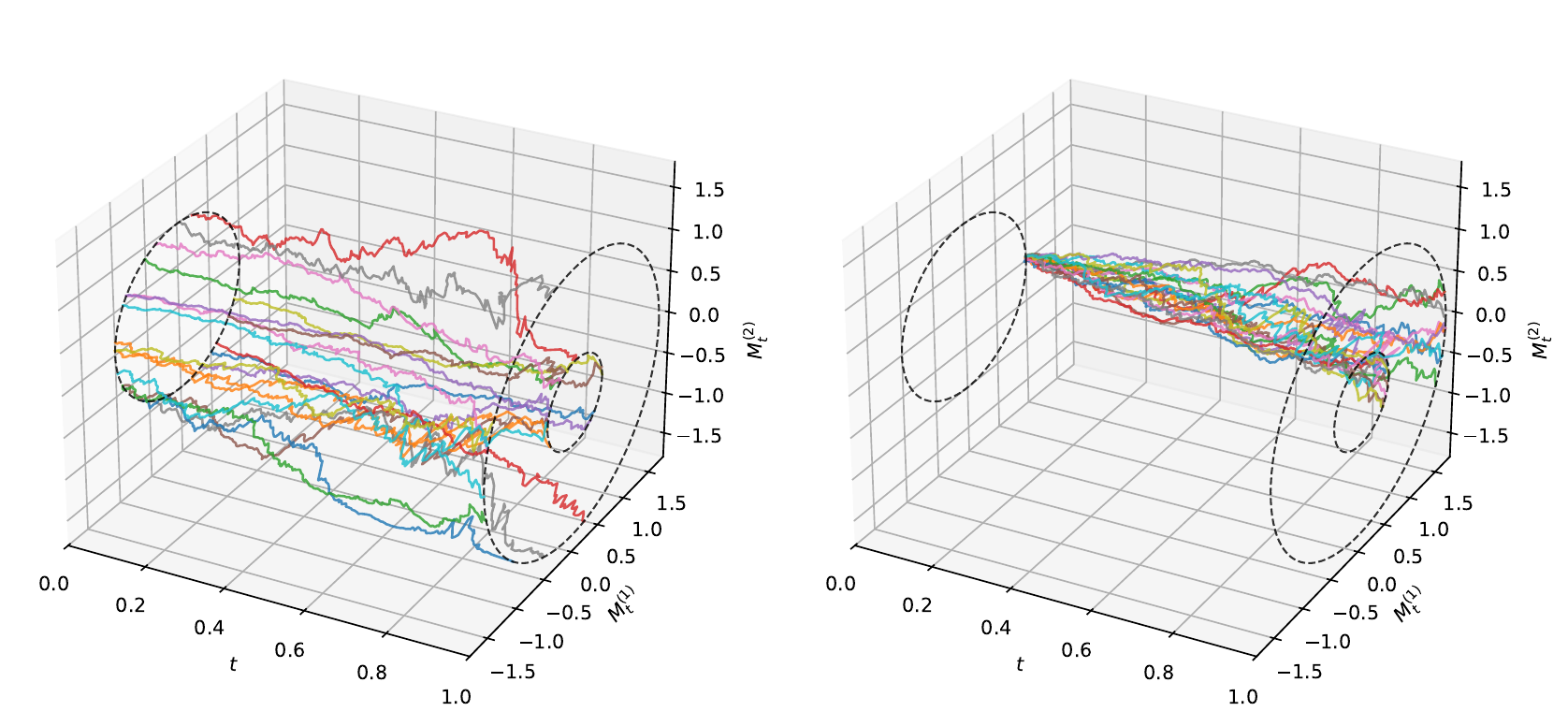}

   \caption{
    Sample trajectories of the Föllmer martingale between $
\mu=\textnormal{Unif}(\mathbb S^1)$ and $
\nu=
\frac12\textnormal{Unif}\!\left(\frac12\mathbb S^1\right)
+
\frac12\textnormal{Unif}\!\left(\frac85\mathbb S^1\right)$ using SISTA.
    }
    \label{fig:2d_follmer}
\end{figure}

\paragraph{Outline}

 \Cref{sec:setting} introduces the entropic martingale optimal transport problem, its dual functional and the standing assumptions of irreducibility and non-degeneracy. 
\Cref{sec:properties} collects the analytic properties of the dual functional that are needed for the convergence analysis. In particular, we establish gradient estimates, as well as strong convexity and smoothness properties. 
\Cref{sec:convergence} is devoted to the convergence analysis of the SISTA scheme. We first show that a weak coercivity property of the dual functional yields uniform boundedness of the iterates, which constitutes the main step of the argument. Once boundedness is established, we can apply standard arguments based on smoothness and strong convexity to derive the linear convergence of the algorithm. 
\Cref{sec:numerics} presents numerical experiments in one and two dimensions and investigates the behavior of the algorithm when irreducibility or non-degeneracy is violated.

\section{Setting and Main Result}\label{sec:setting}

\subsection{Definition of the Problem and Assumptions}
Throughout the theoretical analysis, we consider $\mu,\nu\in\mathcal P(\mathbb R^d)$ in the convex order (i.e. $\mu \leq_c \nu$), which means that for any convex function $u$ finite over the closed convex hull of the support of $\nu$ we have: $\int u \dd \mu \leq \int u \dd \nu$. By \cite{Stra65}, this property ensures the existence of a martingale coupling between $\mu$ and $\nu$.
Let $c\in L^\infty(\mu\otimes\nu)$ be a cost function. Dividing the objective of \labelcref{eq:introEMOT} by $\varepsilon>0$ and replacing $c$ by $c/\varepsilon$, we may assume without loss of generality that the regularization parameter equals one. We thus consider the problem
\begin{equation}\label{eq:emot-setting}
    \inf_{\substack{\pi\in \MT(\mu,\nu)}}  \left\{ \int c(x,y) \dd\pi(x,y)+H(\pi \mid \mu \otimes \nu) \right\}. \tag{EMOT}
\end{equation}
We write $X:=\operatorname{supp}(\mu)$,
$Y:=\operatorname{supp}(\nu)$, and
$\operatorname{conv}(Y)$ for the closed convex hull of $Y$.
In order to ensure the existence of dual potentials which are in $L^\infty(\mu)\times L^\infty(\nu)\times L^\infty(\mu;\R^d)$ we make the following assumptions:
\begin{assumption}
\label{ass:qualification}
\hfill
    \begin{itemize} 
        \item \textbf{Non-degeneracy}: $X,Y$ are compact sets such that $X \subseteq \mathrm{int}(\mathrm{conv}(Y))\neq \emptyset$,
        \item \textbf{Irreducibility}:  $ \exists \pi \in \MT(\mu,\nu)$ such that $\int y\dd \pi_x(y) = x$ and $\pi_x \sim  \nu$ for $\mu-$a.e. $x$,
        \item \textbf{Bounded cost}: $c\in L^\infty(\mu \otimes \nu)$.
    \end{itemize}
\end{assumption}
\begin{remark}
It has been shown in \cite[Theorem 2.9.]{CaMaSy26} that \cref{ass:qualification} is equivalent to the following $W_\infty$ stability property of the pair $(\mu,\nu)$: $\exists \eta > 0$ such that
    \begin{equation*}
        \forall \bar\mu \in \mathcal{P}(\R^d) \; \text{ with } \int x \dd \bar \mu(x) = \int x \dd \mu(x), \qquad  W_\infty(\bar\mu,\mu) \leq \eta \implies \bar\mu\leq_c \nu,
    \end{equation*}
    where $W_\infty$ denotes the Wasserstein-$\infty$ distance (see e.g. \cite{ChaDePe08}). 
    The assumptions of \cite{KaReWa26} differ mainly in two respects: irreducibility is replaced by strict convex order, and $Y$ is assumed to be convex. Recall that strict convex order requires: 
    \begin{equation*}
        \int u \dd \mu < \int u \dd \nu,
    \end{equation*}
    for every non-affine convex function $u$.Under non-degeneracy, irreducibility implies strict convex order, while the converse follows from $W_\infty$-stability, as shown in \cref{app:ass}. This equivalence does not require $Y$ to be convex. Likewise, our analysis does not impose convexity of $Y$, relying instead on geometric properties of its closed convex hull.
\end{remark}
\begin{remark}
    The assumption on the interior of $\mathrm{conv}(Y)$ is equivalent to $X \subseteq \mathrm{ri}(\mathrm{conv}(Y))$, where $\mathrm{ri}$ is the relative interior, up to a linear change of variable and reference space.
\end{remark}
\Cref{ass:qualification} implies the qualification of the moment constraint as shown in \cite[Proposition 2.8., Theorem 2.9., Lemma 2.14.]{CaMaSy26}:
\begin{proposition}\label{prop:qual-constraint}
    Under \cref{ass:qualification} there is $\theta > 0$ and $C_0 > 0$ such that for any $r \in L^\infty(\mu; \R^d)$ satisfying $\int r \dd \mu = 0$ and $\Vert r \Vert_\infty \leq \theta$ there is $\pi \in \Pi(\mu,\nu)$ such that
    \begin{equation*}
         \int y \dd \pi_x(y) = x+r(x) \quad\text{for }\mu\text{-a.e. } x \in X, \quad H(\pi \mid \mu \otimes \nu) \leq C_0.
    \end{equation*}
\end{proposition}
This property first ensures finiteness of \labelcref{eq:emot-setting}, by taking $r=0$. For continuous costs, it also guarantees the existence of bounded dual optimizers \cite[Theorem~4.3]{CaMaSy26}. Here, we only assume $c\in L^\infty(\mu\otimes\nu)$ and establish dual attainment through the convergence of SISTA. Define the dual functional
\begin{equation}
\label{eq:main_dual}
\begin{aligned}
    F:\Hcal&\to \R\cup\{+\infty\} \\
    \xi=(\ph,\psi,h)&\mapsto
\int e^{u_\xi(x,y)}\dd\mu\otimes \nu( x,y)
-
\int \varphi\,\dd\mu
-
\int \psi\,\dd\nu,
\end{aligned}
\end{equation}
with $\xi=(\varphi,\psi,h)$  and
$u_\xi(x,y):=\varphi(x)+\psi(y)  + h(x)\cdot (y-x) - c(x,y)$. 
Though dual attainment under \cref{ass:qualification} is known for continuous costs, in the following we adopt our main convergence result. Indeed, \cref{thm:main-result}, guarantees the existence of a minimizer of $F$ in $\Hcal^\infty$, without using a priori existence of a minimizer, from which we derive strong duality and the Euler--Lagrange conditions in the following proposition. 
\begin{proposition}\label{prop:duality}
    Under \cref{ass:qualification}, \cref{thm:main-result} ensures the existence of $\xi^\ast \in \Hcal^\infty= L^\infty(\mu)\times L^\infty(\nu) \times L^\infty(\mu;\R^d)$ such that $F(\xi^\ast) = \inf_{\Hcal^\infty}F$ and thus
    \begin{equation*}
        \min_{\pi \in \MT(\mu,\nu)} \left\{ \int c \dd \pi+ H(\pi \mid \mu \otimes \nu)\right\} = 1 - \min_{\xi\in \Hcal^\infty}F(\xi).
    \end{equation*}
    Moreover, if $\xi^\ast \in \argmin_{\xi \in \Hcal^\infty} F(\xi)$ we have
    \begin{equation*}
        \{e^{u_{\xi^*}} \mu \otimes \nu \}=\argmin_{\pi \in \MT(\mu,\nu)}\left\{ \int c \dd \pi+ H(\pi \mid \mu \otimes \nu)\right\}
    \end{equation*}
\end{proposition}
\begin{proof}
    Let $\pi \in \MT(\mu,\nu)$ be such that $H(\pi \mid \mu \otimes \nu ) < + \infty$ and $\xi=(\ph,\psi,h) \in \Hcal^\infty$. The Fenchel--Young inequality yields $\mu \otimes \nu$ almost surely
    \begin{equation*}
        \rho(x,y)\log\left(\rho(x,y)\right)-\rho(x,y) + e^{u_\xi(x,y)} \geq u_\xi(x,y)\rho(x,y)
    \end{equation*}
    where $\rho(x,y) = \frac{\dd \pi}{\dd \mu \otimes \nu} (x,y)$. Since $u_\xi \in L^\infty(\mu \otimes \nu)$ and $\rho \in L^1(\mu\otimes \nu)$, because the relative entropy is finite, we have by integration
    \begin{equation*}
        H(\pi \mid \mu \otimes \nu) - 1 + \int e^{u_\xi} \dd \mu \otimes \nu \geq \int \ph \dd \mu + \int \psi \dd \nu + \int h(x)\cdot(y-x) \dd \pi(x,y) - \int c \dd \pi,
    \end{equation*}
    since $\pi \in \Pi(\mu,\nu)$.
    The term involving $h$ is equal to zero because $\int (y-x) \dd \pi_x(y) = 0$ for $\mu$ almost every $x$. Thus we get
    \begin{equation*}
        \int c \dd \pi + H(\pi \mid \mu \otimes \nu) \geq 1-F(\xi).
    \end{equation*}
    Optimizing in $\pi$ and $\xi$ we have
    \begin{equation}\label{eq:proof-weak-duality}
        \inf_{\pi \in \MT(\mu,\nu)}\int c \dd \pi + H(\pi \mid \mu \otimes \nu) \geq 1-\inf_{\xi\in \Hcal^\infty}F(\xi)
    \end{equation}
    \Cref{thm:main-result} ensures the existence of $\xi^\ast\in \Hcal^\infty$ minimizing $F$. Setting $\rho^\ast(x,y) = e^{u_{\xi^\ast}(x,y)}$ we have equality in the Fenchel--Young inequality and thus
    \begin{equation}\label{eq:proof-euler-lagrange}
      1-\min_{\xi \in \Hcal^\infty} F(\xi) = 1-F(\xi^\ast) =   \int c \dd \pi^\ast + H(\pi^\ast \mid \mu \otimes \nu) \geq \inf_{\pi \in \MT(\mu,\nu)} \int c \dd \pi + H(\pi \mid \mu \otimes \nu)
    \end{equation}
    where $\pi^\ast = e^{u_{\xi^\ast}}\mu \otimes \nu$ has finite relative entropy since $u_{\xi^\ast}$ is bounded. The last inequality holds since $\pi^\ast \in \MT(\mu,\nu)$ by the first order optimality conditions at $\xi^\ast$ in the vector space $\Hcal^\infty$. Indeed, using the Fréchet differentiability of $F$ over $\Hcal^\infty$ from \cref{prop:well-defined}, we have $\nabla_\ph F(\xi^\ast) =0$ which implies $\int e^{u_{\xi^\ast}} \dd \nu=1$ and thus that $\pi^\ast$ has first marginal $\mu$. Similarly $\nabla_\psi F(\xi^\ast)=0$ implies that $\pi^\ast$ has second marginal $\nu$. And $\nabla_h F(\xi^\ast) =0$ implies $\int ye^{u_{\xi^\ast}(x,y)} \dd \nu(y)=x$ which is the martingale property for $\pi^\ast$.  
    The conclusion holds by combining \cref{eq:proof-euler-lagrange,eq:proof-weak-duality}. The uniqueness of the minimizer $\pi^\ast$ comes from the strict convexity of the relative entropy.
\end{proof}
This result justifies our goal of minimizing algorithmically $F$ to obtain a solution of \labelcref{eq:emot-setting}. 

\subsection{Definition of the Algorithm and Well Posedness}
Following the entropic optimal transport literature we propose a Sinkhorn based optimization algorithm to minimize $F$ in \cref{eq:main_dual}. This algorithm, introduced in \cite{CaDuGaSu23} as SISTA, consists in two exact block optimization steps and a block gradient step that we detail in \cref{alg:SISTA}.

\begin{algorithm}[ht]
\caption{SISTA}
\label{alg:SISTA}

\begin{algorithmic}[1]

\State \textbf{Input:}
measures $\mu, \nu$, step size $\tau > 0$,
initial $\xi^0 = (\varphi^0, \psi^0, h^0)$

\State \textbf{Define recursively:}

\Statex 
\begin{equation*}
    \xi^{n+\frac{1}{3}}:=(\ph^n,\psi^{n+1},h^n),\quad\xi^{n+\frac{2}{3}}:=(\ph^{n+1},\psi^{n+1},h^n)
\end{equation*}
\State \textbf{(S)} \textit{Sinkhorn block:}

\Statex
\begin{equation*}
    \psi^{n+1}(y)\leftarrow-\log\int \exp\!\left(\varphi^n(x) + h^n(x) \cdot (y - x)-c(x,y)\right)\dd \mu(x)
\end{equation*}

\Statex
\begin{equation*}
    \varphi^{n+1}(x)\leftarrow-\log\int \exp\!\left( \psi^{n+1}(y) + h^n(x) \cdot (y - x)-c(x,y)\right)\dd\nu(y)
\end{equation*}

\Statex
\State \textbf{(G)} \textit{Gradient block:}

\Statex
\begin{equation*}
    h^{n+1}(x)\leftarrow h^n(x) - \tau \nabla_h F(\xi^{n+\frac23})(x)
\end{equation*}

The $h$-gradient at $\xi^{n+\frac23}$ is
\[
R^n(x)
:=
\nabla_hF(\xi^{n+\frac23})(x)
=
\int 
(y-x)e^{u_{\xi^{n+2/3}}(x,y)}\dd \nu(y).
\]
\end{algorithmic}
\end{algorithm}

The natural spaces to study algorithms in infinite dimension are Hilbert spaces. However, most differential properties of $F$ are only true when $u_\xi$ is bounded. This is in particular the case when $\xi$ is bounded. Hence we will carry out the analysis on the following two spaces

\begin{equation*}
    \Hcal := L^2(\mu)\times L^2(\nu)\times L^2(\mu;\R^d), \quad \Hcal^\infty:= L^\infty(\mu)\times L^\infty(\nu)\times L^\infty(\mu;\R^d).
\end{equation*}
\begin{proposition}\label{prop:well-defined}
    The function $F$ is twice Fréchet differentiable as a function defined over $\Hcal^\infty$ and $\nabla F \in \Hcal^\infty$. In particular if $\xi^0 \in \Hcal^\infty$ then $\xi^n \in \Hcal^\infty$ for every $n \in \mathbb{N}$ and, thus, \cref{alg:SISTA} is well defined.
\end{proposition}

\begin{proof}
    Since $c\in L^\infty(\mu\otimes\nu)$, the twice Fréchet differentiability of $F$ on $\mathcal H_\infty$ follows directly from the boundedness of the exponent and the fact that $\mu,\nu$ are probability measures. A direct computation shows that
    \begin{equation*}
        \nabla_\ph F(\xi) = \int e^{u_\xi} \dd \nu-1,\quad \nabla_\psi F(\xi) = \int e^{u_\xi}\dd \mu - 1, \quad \nabla_h F(\xi) = \int( y-x )e^{u_\xi(x,y)} \dd \nu(y)
    \end{equation*}
    and every term is bounded.
    If $\xi^n\in \Hcal^\infty$ then the Sinkhorn blocks remain in $\Hcal^\infty$ and thus $\xi^{n+\frac 13},\xi^{n+\frac 23} \in \Hcal^\infty$. In particular the gradient is well defined and bounded by compactness of $X\times Y$, so that $\xi^{n+1} \in \Hcal^\infty$ which proves the claim.
\end{proof}

Finally, we introduce some notation used throughout the analysis. Under our standing assumptions, we set
\[
M_c\coloneqq\|c\|_{L^\infty(\mu\otimes\nu)}, \qquad
D_\mu\coloneqq\operatorname{diam}(X),\qquad
D_{\nu}\coloneqq\operatorname{diam}(Y).
\]
Moreover, the non-degeneracy assumption ensures that the covariance matrix of $\nu$ is positive definite. We denote
\[
\Sigma_\nu\coloneqq\operatorname{Cov}(\nu),\qquad
\lambda_\nu\coloneqq
\inf_{\substack{x\in\R^d\\\|x\|=1}}
x\cdot\Sigma_\nu x>0.
\]
We are now ready to state our main result.

\begin{theorem} \label{thm:main-result}
Assume that \cref{ass:qualification} holds.
Let $(\xi^n)_{n\geq0}$ be the SISTA iterates
starting from $\xi^0\in\Hcal^\infty$,
with step size $\tau>0$ satisfying
$
2\tau D_{\nu}^2\leq1.
$
Set
$
E_n:=F(\xi^n)-\inf_{\Hcal}F.
$
Then, there exist $\alpha>0$ and $q\in(0,1)$
such that:
\begin{enumerate}
    \item[(i)] $E_{n+1}\leq qE_n$
    for all $n\geq0$;

    \item[(ii)] There exists a minimizer
    $\xi^\infty\in\Hcal^\infty$ of $F$
    over $\Hcal$ such that
    \[
    \|\xi^n-\xi^\infty\|_{\Hcal}
    \leq
    \frac{\sqrt{2E_0/\alpha}}{1-\sqrt{q}}\,
    q^{n/2},
    \qquad n\geq0.
    \]
\end{enumerate}
\end{theorem}
We emphasize the constant $q$ depends on  $\mu$, $\nu$, $\tau$, $||c||_{\infty}$ and $\xi^0$. Also note that the step size $\tau$ ensuring convergence of the algorithm doesn't depend on $c$ and in particular on $\eps$.

\section{Properties of $F$}\label{sec:properties}

In this section we show the properties of $F$ which will be useful to show the  convergence of the algorithm. Recall that 
\begin{align*}
    F:\Hcal&\to \R\cup\{+\infty\} \\
    \xi=(\ph,\psi,h)&\mapsto
\int e^{u_\xi(x,y)}\dd\mu\otimes \nu( x,y)
-
\int \varphi\,\dd\mu
-
\int \psi\,\dd\nu,
\end{align*}
with 
$u_\xi(x,y):=\varphi(x)+\psi(y)  + h(x)\cdot (y-x)-c(x,y)$ is twice Fréchet differentiable (\cref{prop:well-defined}). Note that the function is well defined by positivity of the exponential.

\subsection{Coercivity of $F$ in the $h$ Variable}
The first a priori estimate we will need in the algorithm is the coercivity of the functional in $h$.
\begin{lemma}\label{lem:coercivity}
     There are $C,\theta> 0$ such that for every $\xi \in \Hcal^\infty$
    \begin{equation*}
        F(\xi) \geq \theta \inf_{a \in \R^d} \Vert h - a \Vert_{L^1(\mu)} - C \quad \text{and} \quad\inf_{\Hcal^\infty}\in\R.
    \end{equation*}
\end{lemma}

\begin{proof} 
By weak duality, shown in \cref{prop:duality} we have 
    \begin{equation*}
    \begin{aligned}
        F(\xi) \geq \inf_{\varphi\in L^\infty(\mu),\psi\in L^\infty(\nu)} F(\varphi,\psi,h) &\geq \sup_{\pi\in\Pi(\mu,\nu)} \int h(x)\cdot (y-x) \dd \pi(x,y) -  H(\pi \mid \mu \otimes \nu)-\int c \dd \pi \\
        & \geq \sup_{\pi\in\Pi(\mu,\nu)} \int h(x)\cdot (y-x) \dd \pi(x,y) -  H(\pi \mid \mu \otimes \nu)-M_c.
    \end{aligned}
    \end{equation*}
    \Cref{prop:qual-constraint} ensures the existence of $C_0,\theta >0$  such that for any $r \in L^\infty(\mu; \R^d)$ satisfying $\int r \dd \mu = 0$ and $\Vert r \Vert_\infty \leq \theta$, there exists $\pi \in \Pi(\mu,\nu)$ satisfying $H(\pi\mid \mu \otimes \nu) \leq C_0$ and $\int (y-x) \dd \pi_x(y) = r(x)$. Thus we have
    \begin{equation*}
        F(\xi) \geq \sup_{\Vert r \Vert_{L^\infty(\mu)}\leq \theta, \int r \dd \mu =0} \int h(x)\cdot r(x) \dd \mu(x) - C,
    \end{equation*}
    where, $C:= C_0+M_c$. This
     proves the claim since by \cite[Lemma 6.1.]{CaMaSy26}
    \begin{equation*}
        \sup_{\Vert r \Vert_{L^\infty(\mu)}\leq \theta, \int r \dd \mu =0} \int h(x)\cdot r(x) \dd \mu(x)  = \theta \inf_{a \in \R^d} \Vert h - a \Vert_{L^1(\mu)}.
    \end{equation*}
In particular, since
$F(0,0,0)\leq e^{M_c}<+\infty$,
we have $\inf_{\mathcal H^\infty}F\in\R$.
\end{proof}

\subsection{Properties of $\nabla F$ on $\Hcal^\infty$}

As was pointed out earlier it is not reasonable to expect $F$ to have Lipschitz gradients on the whole of $\Hcal^\infty$. However, when $\ph$ is such that the exponential term has mass one then $\nabla F$ is Lipschitz and bounded.
For $\psi \in L^\infty(\nu),h \in L^\infty(\mu;\R^d)$ define
\begin{equation}\label{eq:normalization-constant}
    Z_{\psi,h}:x\mapsto  -\log\left(\int\exp\left(\psi(y)+h(x)\cdot(y-x) -c(x,y) \right)\dd \nu(y)\right) \in L^\infty(\mu)
\end{equation}
The first observation is that the gradient in $h$ of $F$ is Lipschitz with respect to the $L^2(\mu;\R^d)$ for perturbations which are close in $L^\infty(\mu;\R^d)$.
\begin{proposition}\label{prop:lipschitz-gradients}
Let $\xi = (\ph,\psi,h)\in \Hcal^\infty$ be such that $\ph = Z_{\psi,h}$ and $h' \in L^\infty(\mu;\R^{d})$ then
    \begin{equation*}
       \Vert \nabla_hF(Z_{\psi,h},\psi,h)-\nabla_h F(Z_{\psi,h},\psi,h')\Vert_{L^2(\mu;\R^d)} \leq D_{\nu}^2 e^{\Vert h-h'\Vert_{L^\infty(\mu;\R^) }D_{\nu}} \Vert h-h'\Vert_{L^2(\mu;\R^d)}
    \end{equation*}

\end{proposition}

\begin{proof}
    First recall the expression of the gradient for $\tilde h \in L^\infty(\mu;\R^d)$
    \begin{equation*}
        \nabla_hF(Z_{\psi,h},\psi,\tilde h)(x) = \int (y-x) \exp(Z_{\psi,h}(x)+\psi(y)+\tilde h(x)\cdot(y-x)-c(x,y))\dd \nu(y).
    \end{equation*}
    Introduce for almost every $x$ the following probability measure $\dd \pi_x(y) = \exp(Z_{\psi,h}(x)+\psi(y)+h(x)\cdot(y-x)-c(x,y))\dd \nu(y)$. Note that $\pi_x$ is supported on $Y$. Let $x \in X$ we have
    \begin{align*}
        &\Vert \nabla_hF(Z_{\psi,h},\psi,h)(x)-\nabla_h F(Z_{\psi,h} ,\psi, h')(x)\Vert^2\\
        &= \left\Vert \int (y-x)[1-\exp((h'(x)-h(x))\cdot(y-x))] \dd \pi_x(y) \right\Vert^2\\
        &\leq \int \Vert y-x \Vert^2 [1-\exp((h'(x)-h(x))\cdot(y-x))] ^2 \dd \pi_x(y)\\
        &\leq D_{\nu}^2 e^{2\Vert h-h'\Vert_{L^\infty(\mu)}D_{\nu}}D_{\nu}^2\Vert h'(x)-h(x)\Vert^2.
    \end{align*}
    The second inequality holds by Jensen's inequality and the last one by using the fact that $\pi_x$ is supported on $Y$ and by bounding the derivative of the exponential over the segment $[0,(h'(x)-h(x))\cdot(y-x)]$. Integration in the $x$ variable with respect to $\mu$ yields the result.
\end{proof}
The second observation is that the normalization ensures the boundedness of the gradients in the whole of $\Hcal^\infty$.
\begin{proposition}\label{prop:bounded-gradients}
     Let $\xi = (\ph,\psi,h)\in \Hcal^\infty$ be such that $\ph = Z_{\psi,h}$ then
    \begin{equation*}
        \Vert \nabla_h F(Z_{\psi,h},\psi,h)\Vert_{L^\infty(\mu;\R^d)} \leq D_{\nu}
    \end{equation*}
\end{proposition}

\begin{proof}
    A direct computation gives
    \begin{equation*}
        \nabla_h F(Z_{\psi,h},\psi,h)(x) =\int  (y-x) \exp( Z_{\psi,h}(x)+ \psi(y) + h(x)\cdot(y-x)-c(x,y)) \dd \nu(y).
    \end{equation*}
    Since $y-x$ is bounded by $D_{\nu}$ over $Y$ and the exponential in the integrand has integral one we can conclude.
\end{proof}

\subsection{A Pointwise Coercivity Estimate for the $h$-Gradient}
We establish a coercivity estimate for the $h$-gradient of $F$, which is crucial for proving the uniform boundedness of the iterates in \cref{prop:Linfty}. We begin by deriving an upper bound on the normalization constant $Z_{\psi,h}$ in terms of $h$ and $\|\psi\|_{L^1(\nu)}$.
\begin{lemma}\label{lem:upper-bound-zpsih}
    Let $\psi\in L^\infty(\nu),h\in L^\infty(\mu;\R^d)$. There is $\rho>0$ depending only on $\mu,\nu$ such that $\mu$ almost surely
    \begin{equation*}
        Z_{\psi,h}(x) + \rho \Vert h (x)\Vert \leq  \frac{\Vert \psi \Vert_{L^1(\nu)}+M_c}{\rho}-\log(\rho).
    \end{equation*}
\end{lemma}

\begin{proof}
If $\mu$ is a Dirac mass the following claim is direct since $\Sigma_\nu$ is invertible. Else by \cite[Lemma 2.11.]{CaMaSy26}, there exists a constant
\(\rho\in(0,1]\) such that, for every \(x\in X\) and every
\(e\in\mathbb S^{d-1}\), there exists a Borel set
\(A=A(x,e)\) satisfying
\begin{equation}
    \nu(A)\ge \rho>0,
    \qquad
    e\cdot(y-x)\ge \rho
    \quad\text{for every }y\in A.
\label{eq:directional_non_degeneracy}
\end{equation}

Fix $x\in X$, on $ A:=A\left (x,\frac{h(x)}{\Vert h(x)\Vert} \right)$ we have
\[
e^{-Z_{\psi,h}(x)}
=
\int 
\exp\!\left(
\psi(y)
+
 h(x)\cdot(y-x)-c(x,y)
\right)\dd \nu(y)
\ge
e^{\rho \Vert h(x) \Vert}
\int_{A} e^{\psi(y)-c(x,y)}\dd \nu(y).
\]
By Jensen's inequality,
\begin{equation*}
\int_A e^{\psi(y)-c(x,y)}\,\dd\nu \ge
\nu(A)
\exp\!\left(
\frac{1}{\nu(A)}
\int_A\psi(y)-c(x,y)\,\dd\nu
\right)
\ge
\rho
\exp\!\left(-\frac{\Vert \psi\Vert_{L^1(\nu)}+M_c}{\rho}\right).
\end{equation*}
Consequently,
\[
e^{-Z_{\psi,h}(x)}
\ge
\rho
\exp\!\left(
\rho \Vert h(x) \Vert-\frac{\Vert \psi\Vert_{L^1(\nu)}+M_c}{\rho}
\right).
\]
Taking the logarithm concludes the proof.
\end{proof}
This upper bound on the normalization constant ensures the following weak coercivity of $F$. 
\begin{proposition}\label{prop:pullback}
    Let $\psi \in L^\infty(\nu),h\in L^\infty(\mu;\R^d)$. Define $M(a,b,\rho)= a + \frac{b}{\rho}-\log(\rho)$ and $C(a,b,\rho) $ equal to
    \begin{equation*}
         \inf\left\{x >0 \,: \forall y\geq x,\,  \frac{\rho }{2}y\left(1-\exp\left(M(a,b,\rho)-\frac{\rho}{2}y\right)\right)-\exp\left(M(a,b,\rho) - \rho y\right)\geq \frac{\rho}{4} y \right\},
    \end{equation*}
    which is finite for $\rho > 0$.
    Then, 
        $\Vert h(x) \Vert \geq C(\Vert \psi^+\Vert_{L^\infty(\nu)}+M_c,\Vert \psi\Vert_{L^1(\nu)}+M_c,\rho)$ implies
     \begin{equation}\label{eq:inequality_PL}
          h(x)\cdot \nabla_hF(Z_{\psi,h},\psi,h)(x)\geq \frac{\rho}{4} \Vert h(x) \Vert.
    \end{equation}
\end{proposition}

\begin{proof}
    Set $\dd \pi_x(y) = \exp(Z_{\psi,h}(x)+\psi(y)+h(x) \cdot (y-x)-c(x,y))\dd \nu(y)$. 
    By definition of the gradient of $F$ we have
    \begin{equation*}
        h(x) \cdot \nabla_hF(Z_{\psi,h},\psi,h)(x) = \int h(x) \cdot (y-x)  \dd \pi_x(y).
    \end{equation*}
 Let $A_t=\{y : h(x)\cdot(y-x) \leq t\}$ and observe that for $t\in \R$  we have
    \begin{equation}
        \pi_x(A_t) \leq \exp(t + \Vert \psi^+\Vert_{L^\infty(\nu)} +M_c+Z_{\psi,h}(x)).
    \end{equation}
    Setting $M(a,b,\rho) : = a + \frac{b}{\rho}-\log(\rho)$, \cref{lem:upper-bound-zpsih} ensures
    \begin{equation}\label{eq:control-repartition-func}
        \pi_x(A_t) \leq \exp(t +M(\Vert\psi^+\Vert_{L^\infty(\nu)}+M_c,\Vert \psi\Vert_{L^1(\nu)}+M_c,\rho)-\rho \Vert h(x)\Vert).
    \end{equation}
    Thus by integration by parts we can control the negative part of the integral
    \begin{equation}\label{eq:control-neg-part}
    \begin{aligned}
        \int  &1_{h(x)\cdot(y-x)\leq 0}h(x)\cdot(y-x) \dd \pi_x(y) \geq -\int_{-\infty}^0 \pi_x(A_t)\dd t\\
        &\geq -\exp\left( M(\Vert\psi^+\Vert_{L^\infty(\nu)}+M_c,\Vert \psi\Vert_{L^1(\nu)}+M_c,\rho)- \rho \Vert h(x) \Vert\right)
    \end{aligned}
    \end{equation}
    where we used \cref{lem:upper-bound-zpsih} for the last inequality. On the other hand, from \cref{eq:control-repartition-func} for $t \geq 0$ 
    we get
    \begin{equation}\label{eq:control-pos-part}
    \begin{aligned}
        \int  1_{h(x)\cdot(y-x)\geq 0} h(x)\cdot(y-x) \dd \pi_x(y) &\geq  \int  1_{h(x)\cdot(y-x)\geq t} h(x)\cdot(y-x) \dd \pi_x(y)\\
        &\geq t\pi_x(\{h(x)\cdot(y-x)\geq t\})
        \geq t(1-\pi_x(A_t)).
    \end{aligned}
    \end{equation}
    In particular taking $t=\frac{\rho}{2}\Vert h(x)\Vert$ and combining \cref{eq:control-neg-part,eq:control-pos-part,eq:control-repartition-func} we get
    \begin{align*}
    h(x)\cdot \nabla_hF(Z_{\psi,h},\psi,h)&(x)
    \geq \\
    &\frac{\rho}{2}\Vert h(x)\Vert
    \Bigl(1-\exp\Bigl(
    M(\Vert\psi^+\Vert_{L^\infty(\nu)}+M_c,
    \Vert\psi\Vert_{L^1(\nu)}+M_c,\rho)
    -\frac{\rho}{2}\Vert h(x)\Vert\Bigr)\Bigr) \\
    &-\exp\Bigl(
    M(\Vert\psi^+\Vert_{L^\infty(\nu)}+M_c,
    \Vert\psi\Vert_{L^1(\nu)}+M_c,\rho)
    -\rho\Vert h(x)\Vert\Bigr).
    \end{align*}
    Take $C(a,b,\rho)$ as in the statement.
    The constant is not $+\infty$ because $\rho > 0$ and the left hand side of the set defining inequality is equivalent to $\frac{\rho}{2}y$ as $y \to +\infty$. Finally, by definition of $C(a,b,\rho)$ we have 
    \begin{equation*}
        \Vert h(x) \Vert \geq C(\Vert \psi^+\Vert_{L^\infty(\nu)}+M_c,\Vert \psi\Vert_{L^1(\nu)}+M_c,\rho) \implies h(x)\cdot \nabla_hF(Z_{\psi,h},\psi,h)(x)\geq \frac{\rho}{4} \Vert h(x) \Vert.
    \end{equation*}
\end{proof}
\begin{remark}\label{rem:PL_ball}
    Note that $(a,b) \mapsto C(a,b,\rho)$ is non-decreasing. Indeed for $y$ fixed the function involved in the defining inequality is non-increasing in $(a,b)$, because it is non-increasing  in $M(a,b,\rho)$ and $M(a,b,\rho)$ is non-decreasing in $(a,b)$. Thus if $a \ge \Vert \psi^+\Vert_{L^\infty(\nu)}+M_c$, $b \ge \Vert \psi\Vert_{L^1(\nu)}+M_c$ and $\Vert h(x) \Vert \geq C(a,b,\rho)$, then \cref{eq:inequality_PL} still holds.
\end{remark}

\subsection{Strong Convexity and Smoothness}
In this section we show that on subsets of $\Hcal^\infty$ the function $F$ is  strongly convex on the directions orthogonal to its invariance set:
\begin{equation*}
    \Acal
:=
\left\{
(-a\cdot x+b,\,a\cdot y-b,\,-a):
a\in\R^d,\ b\in\R
\right\}
\subset\Hcal^\infty,
\end{equation*}
smooth and satisfies a Polyak--{\L}ojasiewicz inequality. To make these properties invariant by renormalization we introduce for $M<+\infty$ the set
\begin{equation}\label{eq:KM}
    K_M:= \{\xi \in \Hcal \mid \Vert u_\xi\Vert_{L^\infty(\mu \otimes\nu)} \leq M\}.
\end{equation}

The following lemma states some properties of $K_M$.

\begin{lemma}\label{lem:prop_KM}
Let $M \in \R$ and $K_M$ defined by \cref{eq:KM}:
    \begin{enumerate}
        \item $K_M$ is a closed convex subset of $\mathcal H$ on which $F$ is finite;
        \item $K_M \subset \mathcal H^\infty$.
    \end{enumerate}
\end{lemma}

\begin{proof}
    We start by the proof of the first point. Let $T : \mathcal H \to L^2(\mu \otimes \nu)$ be defined by $T(\xi) = u_\xi$. We have $K_M = T^{-1}(B_M)$ where $B_M = \{u \in L^2(\mu \otimes \nu)\,  ; \, || u ||_{L^\infty(\mu \otimes \nu)} \le M\}$. As $B_M$ is convex and $T$ is affine, $K_M$ is convex. Moreover, for all $\eta, \xi \in \mathcal H$, $||T(\xi)-T(\eta)||_{L^2(\mu \otimes \nu)} \le 3\max(1,D_{\nu})||\xi-\eta||_{\mathcal H}$, which proves that $T$ is continuous. As $L^2(\mu \otimes \nu)$ convergence yields $\mu \otimes \nu$-a.e. convergence along a subsequence, $B_M$ is a closed subset of $L^2(\mu \otimes \nu)$. Together with the continuity of $T$, this proves that $K_M$ is closed. The fact that $F$ is finite follows directly by definition of $K_M$: for all $\xi \in K_M$, $|F(\xi)| \leq e^M + ||\varphi||_{L^2(\mu)} + ||\psi||_{L^2(\nu)}.$ 
    We now prove the second point. Let $\xi \in K_M$. Since \(c\in L^\infty(\mu\otimes\nu)\), we have \(\|u_\xi+c\|_{L^\infty(\mu\otimes\nu)}\leq M+M_c\). Multiplying the equality
    \begin{equation}\label{eq:formule_u}
        u_\xi(x,y) +c(x,y) = \varphi(x) + \psi(y) + h(x)\cdot (y-x),
    \end{equation}
    by $(y-m_\nu)$, where $m_\nu$ is the barycenter of $\nu$, and integrating against $\nu$ yields
    $$\int (u_\xi(x,y)+c(x,y))(y-m_\nu) \dd \nu(y) = \int \psi(y) (y-m_\nu) \dd \nu(y) + \Sigma_\nu  h(x).$$ 
    As for $\mu$-a.e $x \in \R^d$, $|u_\xi(x,y)| \le M$ for $\nu$-a.e. $y \in \R^d$, it holds
    \begin{align*}
        \lambda_\nu \|h(x)\| &\le ||\Sigma_\nu h(x)|| \\ &\le \left| \int (u_\xi(x,y)+c(x,y))(y-m_\nu) \dd \nu(y) \right| + \left| \int \psi(y) (y-m_\nu) \dd \nu(y) \right|\\
        &\le (M+M_c) D_{\nu} +  D_\nu ||\psi||_{L^1(\nu)}.
    \end{align*}
Here we use $\sup_{y\in Y}|y-m_\nu|\leq D_\nu$, since $m_\nu\in\operatorname{conv}(Y)$, and $\lambda_\nu$ denotes the smallest eigenvalue of $\Sigma_\nu$. By \cref{ass:qualification}, $\lambda_\nu$ is positive. This shows  $h \in L^\infty(\mu;\R^d)$. Integrating \cref{eq:formule_u} against $\nu$ yields 
    \begin{equation}
        \int (u_\xi(x,y)+c(x,y)) \dd \nu(y) = \varphi(x) + m_\psi + h(x)\cdot(m_{\nu}-x),
    \end{equation}
    where $m_\psi = \int \psi \dd\nu$. By Cauchy--Schwarz and the triangle inequalities, we get
    \begin{equation}
        |\varphi(x)| \le M +M_c+ |m_\psi| +\Vert h(x)\Vert D_{\nu}. 
    \end{equation}
    As $h \in L^\infty(\mu;\R^d)$, $\varphi \in L^\infty(\mu).$ Similarly, integrating \cref{eq:formule_u} against $\mu$ yields $\psi \in L^\infty(\nu).$
\end{proof}

We first prove the strong convexity orthogonally to the invariance set.
\begin{lemma}\label{prop:strong-convexity}
    Let $M<+\infty$, $\xi \in K_M$ and $\delta \xi \in \mathcal H^\infty$. Then 
    $$D^2F(\xi)[\delta\xi,\delta\xi]
\ge
C_{\mathrm{sc}}(M)\,\|\delta\xi\|_{\mathrm{quot}}^2,
\qquad
C_{\mathrm{sc}}(M):=e^{-M}\kappa,$$
where $\kappa = \min\left\{
\frac12,
\frac{\lambda_\nu}{1+2D_\mu^2}
\right\}$ and $\Vert \xi \Vert_\mathrm{quot} = \inf_{l \in \Acal}\Vert \xi + l \Vert_{\Hcal}$.
\end{lemma}

\begin{proof}
 Let  $\delta\xi=(\delta\varphi,\delta\psi,\delta h) \in \Hcal^\infty$, the Hessian quadratic form of $F$ is
\[
D^2F(\xi)[\delta\xi,\delta\xi]
=
\int
e^{u_\xi(x,y)}
\left|
\delta\varphi(x)
+
\delta\psi(y)
+
\delta h(x)\cdot(y-x)
\right|^2
\dd\mu\otimes \nu(x,y).
\]
Since $\xi\in K_M$, we have
\[
e^{u_\xi(x,y)}\geq e^{-M}>0.
\]
It is therefore sufficient to establish coercivity modulo $\Acal$ of the quadratic form
\[
Q(\delta\xi)
:=
\int
\left|
\delta\varphi(x)
+
\delta\psi(y)
+
\delta h(x)\cdot(y-x)
\right|^2
\dd\mu\otimes \nu(x,y).
\]
Set $v_{\delta \xi}(x,y) = \delta\varphi(x) + \delta\psi(y) + \delta h(x)\cdot(y-x)$. Let $\ell(y)\ :=(y\mapsto a\cdot y +b)\in L^2(\nu)\cap L^2(\mu)$ be the orthogonal projection of $\delta \psi$ on the space of affine functions. A direct computation gives
\begin{equation*}
      a = \Sigma_\nu^{-1} \int  (y-m_\nu) \delta \psi(y) \dd \nu(y),\quad b =  \int  \delta \psi(y) \dd \nu(y)- a\cdot m_\nu.
\end{equation*}

Fix $x$, by orthogonal projection we have 
\begin{align*}
    \int  \vert v_{\delta \xi}(x,y) \vert^2 \dd \nu(y) & = \Vert \delta \psi - \ell \Vert_{L^2(\nu)}^2 + \int  \vert \delta \ph(x) + \delta h(x)\cdot(y-x) + \ell(y) \vert^2 \dd \nu(y)\\
    &=\Vert \delta \psi - \ell \Vert_{L^2(\nu)}^2 + \vert \delta \ph(x)+ b + a \cdot m_\nu+\delta h(x) \cdot(m_\nu-x)\vert^2  \\ & \quad+  (\delta h(x) + a)^\top \Sigma_\nu (\delta h(x) + a).
\end{align*}
By definition of $\lambda_\nu$ we have $ (\delta h(x) + a)^\top \Sigma_\nu (\delta h(x) + a) \geq \lambda_\nu \Vert \delta h(x) + a\Vert^2$. Integration with respect to $x$ yields
\begin{equation*}
    Q(\delta \xi) \geq \Vert \delta \psi - \ell \Vert_{L^2(\nu)}^2 + \Vert \delta \ph+ b + a \cdot m_\nu +\delta h\cdot(m_\nu-\mathrm{Id})\Vert^2_{L^2(\mu)} + \lambda_\nu\Vert\delta h + a\Vert^2_{L^2(\mu;\R^d)}.
\end{equation*}
Define $\widehat{\delta\xi}\in \Hcal$ satisfying $\widehat{\delta\xi}-\delta\xi \in \Acal$ (so that $\Vert \widehat{\delta\xi}\Vert_{\mathrm{quot}} = \Vert \delta\xi\Vert_{\mathrm{quot}}$) by
$
    \widehat{\delta\xi} := (\delta \ph +\ell,\delta \psi-\ell,\delta h + a).
$
Note that \(\delta \ph+ b + a \cdot m_\nu +\delta h\cdot(m_\nu-\mathrm{Id}) = \delta \ph+ b + a \cdot \Id +(\delta h+a)\cdot(m_\nu-\mathrm{Id})\), thus we have
\begin{equation*}
     Q(\delta \xi) \geq \Vert \widehat{\delta \psi} \Vert_{L^2(\nu)}^2 + \Vert \widehat{\delta \ph} +\widehat{\delta h}\cdot(m_\nu-\mathrm{Id})\Vert^2_{L^2(\mu)} + \lambda_\nu\Vert\widehat{\delta h} \Vert^2_{L^2(\mu;\R^d)}.
\end{equation*}
Observe that
\[
\|\widehat{\delta\varphi}\|_{L^2(\mu)}^2
\le
2\| \widehat{\delta \ph} +\widehat{\delta  h}\cdot(m_\nu-\mathrm{Id})\|_{L^2(\mu)}^2
+
2D_\mu^2
\|\widehat{\delta h}\|_{L^2(\mu;\R^d)}^2,
\]
and therefore
\[
\|\widehat{\delta\xi}\|_{\Hcal}^2
\le
2\|\widehat{\delta \ph} +\widehat{\delta  h}\cdot(m_\nu-\mathrm{Id})\|_{L^2(\mu)}^2
+
\|\widehat{\delta\psi}\|_{L^2(\nu)}^2
+
(1+2D_\mu^2)
\|\widehat{\delta h}\|_{L^2(\mu;\R^d)}^2.
\]
Thus
\[
Q(\delta\xi)
\ge
\kappa\|\widehat{\delta\xi}\|_{\Hcal}^2,
\qquad
\kappa
:=
\min\left\{
\frac12,
\frac{\lambda_\nu}{1+2D_\mu^2}
\right\}.
\]
Since
$
\|\delta\xi\|_{\mathrm{quot}}
\le
\|\widehat{\delta\xi}\|_{\Hcal},
$
we conclude that
\[
D^2F(\xi)[\delta\xi,\delta\xi]
\ge
e^{-M}Q(\delta\xi)
\ge
e^{-M}\kappa\|\delta\xi\|_{\mathrm{quot}}^2.
\]
\end{proof}
Using this quantitative convexity estimate we show that $F$ satisfies a Polyak--{\L}ojasiewicz inequality. 

\begin{proposition}\label{prop:pli}
    Let $M<+\infty$. Then, $\forall \xi\in K_M$,
\begin{equation}
F(\xi)-\inf_{\eta \in K_M} F(\eta)
\leq
\frac{1}{2\kappa e^{-M}}
\|\nabla F(\xi)\|_{\Hcal}^2.
\label{eq:PL_quotient}
\end{equation}
\end{proposition}

\begin{proof}
Fix $\xi,\eta \in K_M$, since $K_M$ is convex by \cref{lem:prop_KM}, the segment joining $\xi$ and $\eta$ is in $K_M$. 
By \cref{prop:strong-convexity},
\[
F(\eta)
\geq
F(\xi)
+
\langle\nabla F(\xi),\eta-\xi\rangle_{\mathcal H}
+
\frac{C_{\mathrm{sc}}}2
\|\eta-\xi\|_{\rm quot}^2,
\]
where $C_{\mathrm{sc}}$ can be taken to be $e^{-M}\kappa$.
Since $F$ is invariant under translations by elements of \(\mathcal A\), its gradient is orthogonal to \(\mathcal A\). Therefore,
\[
\left|
\langle\nabla F(\xi),\xi-\eta\rangle_{\mathcal H}
\right|
\leq
\|\nabla F(\xi)\|_{\mathcal H}
\|\xi-\eta\|_{\rm quot}.
\]
and
\[
F(\xi)- F(\eta) \leq
\|\nabla F(\xi)\|_{\mathcal H}
\|\xi-\eta\|_{\rm quot}
-
\frac{C_{\mathrm{sc}}}2
\|\xi-\eta\|_{\rm quot}^2.
\]
Using
$ar-\frac{\sigma}{2}r^2
\leq\frac{a^2}{2\sigma},
\;r\geq0,$
we obtain, for every $\eta\in K_M$,
\[
F(\xi)-F(\eta)
\leq
\frac{1}{2C_{\mathrm{sc}}}
\|\nabla F(\xi)\|_{\mathcal H}^2.
\]
Taking the infimum over $\eta\in K_M$ proves the result.
\end{proof}

The following result shows the smoothness of the function.
\begin{proposition}\label{prop:smoothness-global}
    Let $M<+\infty$ and $\xi \in K_M$ the function $F$ satisfies:
    \begin{equation*}
        \forall \delta\xi \in \Hcal^\infty, D^2F(\xi)[\delta \xi,\delta \xi] \leq C_{\text{sm}}(M) \Vert \delta \xi \Vert_{\Hcal}^2
    \end{equation*}
    with
    \begin{equation*}
        C_\text{sm}(M) = 3\max\{1,D_{\nu}^2\} \exp(M).
    \end{equation*}
    In particular $F$ has Lipschitz gradients on every set $K_M$.
\end{proposition}
\begin{proof}
    Recall that
    \begin{equation*}
        D^2F(\xi)[\delta\xi,\delta\xi]
=
\int
e^{u_\xi(x,y)}
\left|
\delta\varphi(x)
+
\delta\psi(y)
+
\delta h(x)\cdot(y-x)
\right|^2
\dd\mu\otimes \nu(x,y).
    \end{equation*}
    Since $\xi\in K_M$, we have $e^{u_\xi(x,y)}\leq e^M$.  Using this upper bound in the integral we finally have by Cauchy--Schwarz inequality
    \begin{align*}
        D^2F(\xi)[\delta\xi,\delta\xi] &\leq 3 e^M \left(\Vert \delta \ph\Vert_{L^2(\mu)}^2+\Vert \delta \psi\Vert_{L^2(\nu)}^2+D_{\nu}^2\Vert \delta h\Vert_{L^2(\mu;\R^d)}^2\right)
        \leq 3e^M \max\{1,D_{\nu}^2\} \Vert \delta\xi\Vert_{\Hcal}^2.
    \end{align*}
\end{proof}

The components related to the multipliers of the marginal constraint enjoy a genuine strong convexity property in $L^2(\mu),L^2(\nu)$ norms.
\begin{proposition}\label{prop:block-strong-convexity}
    Let $M<+\infty$ and $\xi \in K_M$ and $\delta \ph \in L^\infty(\mu),\delta \psi \in L^\infty(\nu)$ then
    \begin{equation*}
    \begin{aligned}
        &D^2F(\xi)[(\delta \ph,0,0),(\delta \ph,0,0)] \geq C_{\text{bsc}}(M) \Vert \delta \ph \Vert_{L^2(\mu)}^2, 
        \\
        &D^2F(\xi)[(0,\delta \psi,0),(0,\delta \psi,0)] \geq C_{\text{bsc}}(M) \Vert \delta \psi \Vert_{L^2(\nu)}^2,
    \end{aligned}
    \end{equation*}
    with $C_{\text{bsc}}(M) = \exp(-M)$, which is nonincreasing in $M$. In particular $F$ is strongly convex in the $\ph$ and $\psi$ coordinates.
\end{proposition}
\begin{proof}
    Using the expression of the second derivative we get:
    \begin{equation*}
        D^2F(\xi)[(\delta \ph,0,0),(\delta \ph,0,0)] = \int \vert \delta \ph(x)\vert^2 e^{u_\xi(x,y)}\dd\mu\otimes \nu(x,y).
    \end{equation*}
    By definition of $C_{\text{bsc}}(M) $ we have $e^{u_\xi} \geq C_{\text{bsc}}(M)$ and the result follows by integration. The proof is analogous for the second claim.
\end{proof}

\section{Convergence of the SISTA}\label{sec:convergence}
This section is devoted to the proof of  \cref{thm:main-result}. Standard convergence results do not apply directly, since $F$ is neither globally smooth nor strongly convex on $\Hcal$ or $\Hcal^\infty$. 

Our proof proceeds in three steps.
First, we establish an abstract linear convergence theorem for cyclic block-coordinate methods on Hilbert spaces combining exact minimization and gradient steps, under regularity and convexity assumptions restricted to a closed convex set containing all partial iterates. Although based on standard techniques \cite{XuYin13,BeLu13,RaHoLu13}, this result accommodates both infinite-dimensional spaces and local regularity assumptions.
Second, and most importantly, we establish uniform bounds on suitably normalized iterates, ensuring that all partial iterates remain in a set where the required smoothness and convexity estimates hold uniformly.
Finally, we apply the abstract theorem to prove  \cref{thm:main-result}. We conclude by establishing uniform $\Hcal^\infty$ bounds for the original, unnormalized iterates, thereby justifying their use in the numerical experiments.

The function $F$ being invariant with respect to translations in $\Acal$ we use normalized iterates to carry the boundedness analysis. For the partial iterate $\xi^{n+\frac23}=(\ph^{n+1},\psi^{n+1},h^n)$, define
\[
a^n:=\int h^n\,\dd\mu,
\qquad
b^n:=\int \psi^{n+1}\,\dd\nu+a^n\cdot m_\nu,
\]
and 
\begin{equation}\label{eq:def_nomalized}
\left\{
\begin{aligned}
&\widehat h^n(x)
= h^n(x)-a^n\\
&\widehat\psi^{n+1}(y)
= \psi^{n+1}(y)+a^n\cdot y-b^n\\
&\widehat\varphi^{n+1}(x)
= \varphi^{n+1}(x)-a^n\cdot x+b^n
\end{aligned}
\right. .
\end{equation}
These normalized triples are written
\begin{align*}
    \widehat\xi^{n}&:=(\widehat\varphi^{n},\widehat\psi^{n},\widehat h^n)\\
    \widehat\xi^{n+\frac13}&:=(\widehat\varphi^{n},\widehat\psi^{n+1},\widehat h^n)\\
    \widehat\xi^{n+\frac23}&:=(\widehat\varphi^{n+1},\widehat\psi^{n+1},\widehat h^n).
\end{align*}
Observe that $a^n$ and $b^n$ are chosen in order for $\widehat h^n$ and $\widehat{\psi}^{n+1}$ to satisfy
\begin{equation}
\int \widehat h^n\,\dd\mu=0,
\qquad
\int \widehat\psi^{n+1}\,\dd\nu=0.
\label{eq:centering}
\end{equation}
Moreover, as $u_{\widehat\xi^{n+\frac23}}=u_{\xi^{n+2/3}}$ due to $\Acal$ translation invariance, the $\varphi$-update equality 
\begin{equation*}
\varphi^{n+1}(x)
=
-\log\int 
\exp\!\left(
\psi^{n+1}(y)
+
 h^n(x)\cdot(y-x)
 -c(x,y)
\right)\dd \nu(y),
\end{equation*}
rewrites
\begin{equation}\label{eq:hat_varphi}
\widehat\varphi^{n+1}(x)
=
-\log\int 
\exp\!\left(
\widehat\psi^{n+1}(y)
+
\widehat h^n(x)\cdot(y-x)-c(x,y)
\right)\dd \nu(y).
\end{equation}
Similarly, the $\psi$-update equality 
\begin{equation*}
\psi^{n+1}(y)
=
-\log\int 
\exp\left(
\varphi^n(x)
+
 h^n(x)\cdot(y-x)-c(x,y)
\right)\dd \mu(x).
\end{equation*}
rewrites
\begin{equation}\label{eq:hat_psi}
\widehat\psi^{n+1}(y)
=
-\log\int 
\exp\left(
\tilde\varphi^n(x)
+
\widehat h^n(x)\cdot(y-x)-c(x,y)
\right)\dd \mu(x),
\end{equation}
where 
\begin{equation}\label{eq:tilde_normalization}
    \tilde\varphi^n:=\varphi^n-a^n\cdot x+b^n.
\end{equation}
\subsection{A Convergence Result for a Mixed Cyclic Block Coordinate Descent}\label{subsection:abstract}

Let $\mathcal{H}=\mathcal{H}_1\times\cdots\times\mathcal{H}_N$ be a product of real Hilbert spaces, endowed with its usual scalar product $\langle \cdot,\cdot \rangle$ defined by 
$$
\langle (x_1,\dots,x_N),(y_1,\dots,y_N) \rangle = \sum_{k=1}^N \langle x_k, y_k\rangle.
$$
For all $k \in \{1,\dots,N\}$, let $$J_k : x_k \in \mathcal{H}_k \mapsto (0,\dots, 0, x_k, 0,\dots,0) \in \mathcal H$$ denote the canonical injection from $\mathcal H_k$ to $\mathcal H$. The goal is to prove the linear convergence of a cyclic block coordinate descent with gradient step and exact steps for a given function $F:\mathcal{H}\to(-\infty,+\infty]$. 
The blocks are indexed in their update order and partitioned into sets $\mathcal{E}$ and $\mathcal{G}$, corresponding to exact minimization and gradient steps, respectively.
For the gradient steps we denote by $\tau$ the size of the gradient step. 
Let $K$ be a closed convex subset of $\mathcal H$ on which $F$ is finite. For each $\xi \in K$, we assume the existence of a vector $\nabla F(\xi)=(\nabla_1F(\xi),\dots,\nabla_NF(\xi))\in\mathcal{H}$ such that:
\begin{enumerate}[label=(A\arabic*), ref=(A\arabic*)]
\item\label{point:finite_min} $\inf_K F \in \R$;
\item\label{point:smooth}  $\lVert \nabla F(\xi')-\nabla F(\xi)\rVert\le L\,\lVert \xi'-\xi\rVert $ for all $(\xi,\xi')\in K \times K;$
\item\label{point:convex}  $F(\xi')\ge F(\xi)+\langle \nabla F(\xi),\xi'-\xi\rangle$ for all $\xi\in K$ and $\xi'\in\mathcal{H}$;
\item\label{point:PLI} $
F(\xi)-\inf_K F\ \le\ \frac1{2\sigma}\,\lVert \nabla F(\xi)\rVert^2$ for all $\xi\in K;
$
\item\label{point:strong_block_convexity} For all $k\in\mathcal{E}$ and $(\xi,\xi') \in K \times K$ such that $\xi'-\xi\in J_k(\mathcal{H}_k)$, $$F(\xi')\ge F(\xi)+\langle \nabla F(\xi),\xi'-\xi\rangle+\frac m2\,\lVert \xi'-\xi\rVert^2.$$ 
\end{enumerate}
\begin{remark}
    Note that \cref{point:smooth,point:convex} for convex functions are almost $L$-smoothness. The main difference is that we require the Lipschitz behavior only for directions in the affine span of $K$, thus, it does not necessarily imply differentiability.
\end{remark}
\begin{remark}
    For the application of this theorem in \cref{sec:proof_of_convergence} we will use $K= K_M $ defined in \cref{eq:KM} for $M$ sufficiently large. This set has been shown to be closed and convex in \cref{lem:prop_KM}.
\end{remark}
Moreover, assume there exist partial iterates $\xi^{n+\frac kN}$, $n \ge 0$ and $k\in \{ 0,\dots,N-1\}$ such that: 
\begin{enumerate}[label=(I\arabic*), ref=(I\arabic*)]
\item\label{point:belongs} For all $n \ge 0$ and $k \in \{0,\dots,N-1\}$, the partial iterate $\xi^{n+\frac kN}$ belongs to $K$;
\item\label{point:exact_and_gradient} For all $n \ge 0$ and $k \in \{1, \dots, N\}$, $\xi^{n+\frac kN}   - \xi^{n+ \frac{k-1}{N} }\in J_k\left(\mathcal H_k\right)$ and
\begin{equation}
    \begin{cases}
        \xi^{n+\frac kN}   = \xi^{n+ \frac{k-1}{N}}-\tau\,J_k\left(\nabla_kF(\xi^{n+\frac{k-1}{N}})\right)&\text{if }  k \in \mathcal G \\
        \nabla_k F(\xi^{n+\frac kN}) = 0 & \text{if }k \in \mathcal E
    \end{cases}.
\end{equation}  
\item\label{point:lipschitz_on_segments} For all $k \in\mathcal{G}$, $n \ge 0$ and $s\in [0,1]$, writing $\xi^{n+\frac kN}_s:=\xi^{n+\frac{k-1}{N}}+s\bigl(\xi^{n+\frac {k}{N}}-\xi^{n+\frac {k-1}{N}}\bigr)$, it holds
\[
\lVert \nabla_kF(\xi^{n+\frac kN}_s)-\nabla_kF(\xi^{n+\frac{k-1}{N}})\rVert\le \tau^{-1}\,\lVert \xi^{n+\frac kN}_s-\xi^{n+\frac {k-1}{N}}\rVert.
\]
\end{enumerate}
Though these properties of the iterates are stated abstractly they are exactly mimicking SISTA (without the proximal step). We distinguish the global Lipschitz constant
$L$ from the step-size bound $\tau^{-1}$, since the $h$-gradient admits a sharper estimate along the gradient-update segments.
\begin{theorem}\label{thm:abstract}
Assume \cref{point:finite_min}-\cref{point:strong_block_convexity} and \cref{point:belongs}-\cref{point:lipschitz_on_segments}. Set $E_n=F(\xi^n)-\inf_{K} F \in [0,+\infty)$, because $\xi^n \in K$, and
\[
\alpha=\min\Bigl\{m,\tau ^{-1}\Bigr\},\qquad C=(\tau^{-1}+L\sqrt{N})^2,\qquad q=\frac{C}{C+\sigma\alpha}\in(0,1).
\]
Then $\min_{\mathcal H} F = \min_{ K} F=\inf_K F$ and the following holds:
\begin{enumerate}
\item[(i)] $E_{n+1}\le q\,E_n$ for all $n\ge0$;
\item[(ii)] There exists a minimizer $\xi^\infty$ of $F$ over $\Hcal$ that belongs to $K$ and satisfies
\begin{equation}\label{eq:conv_rate}
    \lVert \xi^n-\xi^\infty\rVert \leq \dfrac{\sqrt{2E_0/\alpha}}{1-\sqrt q}\;  (\sqrt q)^n, \qquad \text{for all } n\ge 0.
\end{equation}
\end{enumerate}
\end{theorem}

\begin{proof}
Given $n \ge 0$ and $k\in \{1, \dots,N\}$, as $\xi^{n+\frac kN}   - \xi^{n+ \frac{k-1}{N} }\in J_k\left(\mathcal H_k\right)$ by \cref{point:exact_and_gradient}, there exists $\Delta_k^n \in \mathcal H_k$ such that $  \xi^{n+\frac {k}{N}} - \xi^{n+\frac {k-1}{N}} = J_k(\Delta_k^n)$. Thus
\begin{equation*}
    \xi^{n+1} - \xi^{n+\frac {k}{N}} = \sum_{j=k+1}^N \xi^{n+\frac {j}{N}} - \xi^{n+\frac {j-1}{N}} = (0,\dots,0, \Delta_{k+1}^n,\dots,\Delta_N^n),
\end{equation*}
which gives
\begin{equation}\label{eq:norme_diff_partial}
    || \xi^{n+1} - \xi^{n+\frac {k}{N}}||^2 = \sum_{l=k+1}^N ||\Delta_l^n||^2.
\end{equation}
For $k=0$, we get 
\begin{equation}\label{eq:norm_diff_global}
    || \xi^{n+1} - \xi^{n}||^2 =\sum_{l=1}^N ||\Delta_l^n||^2 \geq || \xi^{n+1} - \xi^{n+\frac {k}{N}}||^2.
\end{equation}
We will show that
\begin{equation}\label{eq:diff_value_controls_gap}
    \frac\alpha 2\,\lVert \Delta^n_k\rVert^2 \le F(\xi^{n+\frac {k-1}{N}})-F(\xi^{n+\frac {k}{N}}).
\end{equation}
\paragraph{First case $k \in \mathcal G$.}
By \cref{point:belongs} and convexity of $K$,  
\begin{equation}\label{eq:seg_belong}
    \xi_s := (1-s)\xi^{n+\frac {k-1}{N}} + s \xi^{n+\frac {k}{N}} \in K, \quad \text{for all } s\in [0,1].
\end{equation}
We define $$g: s \mapsto F(\xi_s) \qquad \text{and} \qquad \gamma:s \mapsto \langle \nabla F(\xi_s),\xi_1-\xi_0\rangle.$$ By \cref{eq:seg_belong} and \cref{point:convex}, it holds $g(t)\ge g(s)+(t-s)\gamma(s)$ for all $s,t\in [0,1]$. Thus, $g$ is convex and, for $s<t$, $\gamma(s)\le g'_+(s)\le g'_-(t)\le\gamma(t)$. As $\gamma$ is continuous by \cref{point:smooth}, $g$ is differentiable and $g' = \gamma.$ Since $\xi_1-\xi_0 = J_k(-\tau \nabla_kF(\xi_0))$ by \cref{point:exact_and_gradient}, we have $\Delta_k^n = - \tau \nabla_kF(\xi_0)$ and thus$$\gamma(0) = \langle \nabla F(\xi_0),\xi_1-\xi_0\rangle = -\tau ||\nabla_kF(\xi_0)||^2=-\tau^{-1}||\Delta_k^n||^2.$$
Moreover,
$$\gamma(s)-\gamma(0) = \langle \nabla F(\xi_s)-\nabla F(\xi_0),\xi_1-\xi_0\rangle$$
and  by \cref{point:lipschitz_on_segments} and Cauchy--Schwarz's inequality we have
\begin{align*}
    F(\xi^{n+\frac {k}{N}}) - F(\xi^{n+\frac {k-1}{N}}) &= g(1)-g(0) = \gamma(0) + \int_0^1 (\gamma(s) - \gamma(0)) \dd s \\
    &= -\tau^{-1} ||\Delta_k^n||^2 + \int_0^1 \langle \nabla F(\xi_s)-\nabla F(\xi_0),\xi_1-\xi_0\rangle \dd s \\
    &\le -\tau^{-1} ||\Delta_k^n||^2 + \int_0^1 \tau^{-1}s ||\xi_1 - \xi_0||^2 \dd s\\
    &= -\tau^{-1} ||\Delta_k^n||^2 + \frac {\tau^{-1}}{2} ||\Delta_k^n||^2 =  -\frac{1}{2\tau}||\Delta_k^n||^2,
\end{align*}
which yields 
\begin{equation}\label{eq:controle_G}
\frac\alpha 2\,\lVert \Delta^n_k\rVert^2 \le F(\xi^{n+\frac {k-1}{N}})-F(\xi^{n+\frac {k}{N}}).
\end{equation}
\paragraph{Second case $k \in \mathcal E$.}
By \cref{point:strong_block_convexity}, \cref{point:belongs} and \cref{point:exact_and_gradient}, we have 

\begin{equation}\label{eq:controle_K}
\begin{aligned}
    F(\xi^{n+\frac {k-1}{N}})-F(\xi^{n+\frac {k}{N}}) &\ge \langle \nabla F(\xi^{n+\frac {k}{N}}) , \xi^{n+\frac {k-1}{N}} - \xi^{n+\frac {k}{N}} \rangle + \frac m2 ||\xi^{n+\frac {k-1}{N}} - \xi^{n+\frac kN}||^2\\
    &= \frac m2 ||\xi^{n+\frac {k-1}{N}} - \xi^{n+\frac {k}{N}}||^2
    \ge \frac \alpha 2 ||\Delta_k^n||^2.
\end{aligned}
\end{equation}
In any case \cref{eq:diff_value_controls_gap} holds and summing the inequality over $k \in \{1,\dots,N\}$ yields
\begin{equation}\label{eq:controle_global}
    \frac\alpha2\,\lVert \xi^{n+1}-\xi^n\rVert^2 \le F(\xi^n)-F(\xi^{n+1}).
\end{equation}
We will now show that
\begin{equation}\label{eq:controle_of_gradient}
    ||\nabla F(\xi^{n+1})||^2 \le C  ||\xi^{n+1}-\xi^n||^2.
\end{equation}
Let $k \in \{1, \dots, N\}$.

\paragraph{First case $k \in \mathcal G$.}  By \cref{point:exact_and_gradient}, $\nabla_k F(\xi^{n+\frac {k-1}{N}}) = - \tau^{-1} \Delta_k^n$. Together with \cref{point:smooth,eq:norm_diff_global} this yields
\begin{align*}
    ||\nabla_kF(\xi^{n+1})|| &\le ||\nabla_kF(\xi^{n+\frac{k-1}{N}}) ||+||\nabla_kF(\xi^{n+1}) - \nabla_kF(\xi^{n+\frac{k-1}{N}}) ||\\
    &\le \tau^{-1}||\Delta_k^n||+ L||\xi^{n+1}-\xi^{n+\frac{k-1}{N}}||\\
    &\le \tau^{-1}||\Delta_k^n||+ L||\xi^{n+1}-\xi^{n}||.
\end{align*}
\paragraph{Second case $k \in \mathcal E$.} By \cref{point:exact_and_gradient} we have $\nabla_k F(\xi^{n+\frac {k}{N}}) = 0$. Thus \cref{point:smooth,eq:norm_diff_global} yield with the same reasoning
\begin{equation}
    ||\nabla_kF(\xi^{n+1})|| \le \tau^{-1}||\Delta_k^n||+ L||\xi^{n+1}-\xi^{n}||.
\end{equation}
Taking the squares and summing over $k \in \{1, \dots, N\}$, Minkowski's inequality in $\R^N$, together with \cref{eq:norm_diff_global} gives
\begin{align*}
    ||\nabla F(\xi^{n+1})|| &= \left(\sum_{k=1}^N ||\nabla_kF(\xi^{n+1})||^{2}\right)^{1/2} 
    \le \left(\sum_{k=1}^N (\tau^{-1}||\Delta_k^n||)^{2}\right)^{1/2} + \left(\sum_{k=1}^N (L ||\xi^{n+1}-\xi^{n}||)^{2}\right)^{1/2} \\
    &\leq (\tau^{-1} + L \sqrt{N}) ||\xi^{n+1}-\xi^n||,
\end{align*}
which rewrites
\begin{equation*}
    ||\nabla F(\xi^{n+1})||^2 \le C  ||\xi^{n+1}-\xi^n||^2.
\end{equation*}
By \cref{point:PLI} at $\xi^{n+1} \in K$, $$E_{n+1} \le \frac1{2\sigma} ||\nabla F(\xi^{n+1})||^2.$$
Together with \cref{eq:controle_global} and \cref{eq:controle_of_gradient}, this yields
\begin{equation}
    E_{n+1} \le \frac1{2\sigma} ||\nabla F(\xi^{n+1})||^2 \le \frac{C}{2\sigma}||\xi^{n+1}-\xi^n||^2 \le \frac{C}{\alpha \sigma} (E_n-E_{n+1}),
\end{equation}
which rewrites $E_{n+1} \le qE_{n}$ and proves the first point. To prove the second point, note that by induction
\begin{equation}\label{eq:GR}
    E_n \le E_0 q^n, \qquad \text{for all } n \ge 0.
\end{equation}
Together with \cref{eq:controle_global}, this yields $$||\xi^{n+1} - \xi^n|| \le \sqrt{\frac 2\alpha (E_n-E_{n+1})} \le \left(\frac{2E_0}{\alpha}\right)^{1/2}q^{n/2}.$$ By the triangle inequality,
\begin{equation}\label{eq:Cauchy}
\lVert \xi^m-\xi^n\rVert\  \le\ \sum_{j\ge n}\lVert \xi^{j+1}-\xi^j\rVert\ \le\ \Bigl(\frac{2E_0}\alpha\Bigr)^{1/2}\sum_{j\ge n}q^{j/2}\ =\ \frac{\sqrt{2E_0/\alpha}}{1-\sqrt q}\;q^{n/2},
\end{equation}
which proves $(\xi^n)_{n\ge 0}$ is a Cauchy sequence, therefore  converging to an element $\xi^\infty$. As $K$ is closed, $\xi^\infty$ belongs to $K$. The convergence rate \cref{eq:conv_rate} directly follows by taking the limit in $m$ in \cref{eq:Cauchy}. We are left to show that $\xi^\infty$ minimizes $F$. By \cref{eq:controle_of_gradient} $\nabla F(\xi^n) \to 0$, while \cref{point:smooth} proves $\nabla F$ is continuous on $K$. This yields $\nabla F(\xi^\infty) = 0$, which by \cref{point:convex} shows $\xi^\infty$ globally minimizes $F$. 
\end{proof}

The rest of the section is dedicated to proving that we can apply this result to our dual functional. The properties \cref{point:smooth}--\cref{point:strong_block_convexity} and \cref{point:lipschitz_on_segments} are straightforward consequences of the properties of $F$ shown in \cref{sec:properties}, while \cref{point:exact_and_gradient} will directly follow from the definition of the SISTA scheme. The only difficult part, carried out in the next subsection, is to show \cref{point:belongs}. For this, we prove that the normalized iterates $\hat{\xi}^n$ are in fact contained in a $\mathcal H^\infty$- ball. As any $\mathcal H^\infty$ ball is contained in a $K_M$ set and $K_M$ sets are $\mathcal A$-invariant, this concludes.

\subsection{Boundedness of the Iterates}\label{sec:bounded-iterates}

Throughout this section let $\varphi^n, \psi^n,h^n$  and $\xi^{n},\xi^{n+\frac13}, \xi^{n+\frac23}$ be defined as in \cref{alg:SISTA}, assume also that $2\tau D_{\nu}^2 \leq 1$. We start by proving that our objective function $F$ is non-increasing along the partial iterates, and give a quantitative estimate for the gradient block update.

\begin{lemma}\label{lem:descent}
Assume that $2\tau D_{\nu}^2 \leq 1$, then
\begin{equation}\label{eq:stric_grad_decrease}
F(\xi^{n+1})
\le
F(\xi^{n+\frac23})
-
\frac{\tau}{2}
\|\nabla_hF(\xi^{n+\frac23})\|_{L^2(\mu;\R^d)}^2,
\end{equation}
and the objective is non-increasing along the whole iteration,
\begin{equation}
F(\xi^{n+1})
\le
F(\xi^{n+\frac23})
\le
F(\xi^{n+\frac13})
\le
F(\xi^n).
\label{eq:monotonicity}
\end{equation}
\end{lemma}

\begin{proof}
The exact $\psi$ and
$\varphi$-updates are block minimizations, hence
$F(\xi^{n+\frac23})\le F(\xi^{n+\frac13})\le F(\xi^n)$. To prove \cref{eq:stric_grad_decrease}, set $$\xi_s =\xi^{n+2/3} + s (\xi^{n+1}-\xi^{n+2/3}) = (\varphi^{n+1},\psi^{n+1}, h^n - s \tau R^n)$$ and $g(s) = F(\xi_s)$ for all $s \in [0,1]$. It holds, $g'(s) = \langle \nabla_h F(\xi_s),-\tau R^n \rangle,$ and in particular $g'(0) = - \tau \| R^n \|^2$. Thus 
\begin{equation}\label{eq:proof-decrease}
\begin{aligned}
g(1)-g(0) &= \int_0^1 \langle \nabla_h F(\xi_s),-\tau R^n \rangle \, \dd s \\
&= - \tau \| R^n \|^2 + \int_0^1 \langle \nabla_h F(\xi_s)- \nabla_h F(\xi_0),-\tau R^n\rangle \, \dd s.
\end{aligned}
\end{equation}
Since $\xi_s-\xi_0=(0,0,-s\tau R^n)$ has norm $s\tau\|R^n\|$, the
Cauchy--Schwarz inequality and the Lipschitz bound given by \cref{prop:lipschitz-gradients} applied to \cref{eq:proof-decrease} bound the integrand by $L(\tau)\,s\tau^2\|R^n\|^2$ with $L(\tau) =  D_{\nu}^2 e^{\tau D_{\nu}^2}$. Integrating in $s$ gives $$F(\xi^{n+1})-F(\xi^{n+\frac23})
\le
-\tau\Bigl(1-\frac{\tau L(\tau)}{2}\Bigr)\|R^n\|^2 .$$ As $\tau L(\tau)= \tau  D_{\nu}^2 e^{\tau D_{\nu}^2}\le 1$, since $\tau D_{\nu}^2 \leq \frac{1}{2}$ and $x\leq \frac{1}{2}$ imply $x e^x \leq 1$ , this concludes.
\end{proof}

Throughout this section, we denote by  $(\widehat\varphi^{n+1},\widehat\psi^{n+1},\widehat h^n)$ and $\tilde\varphi^{n}$ the normalized iterates defined in \cref{eq:def_nomalized} and \cref{eq:tilde_normalization}.

\begin{lemma}\label{lem:control_integrals_ph}
    Set $F^\ast = \inf_{\mathcal{H}^\infty} F$, then
    \begin{equation}
1-F(\xi^0)
\le
\int \widehat\varphi^{n+1}\,\dd\mu
\le
1-F^\ast,
\label{eq:mean_phi_bounds}
\end{equation}
\begin{equation}
1-F(\xi^0)
\le
\int \tilde\varphi^{n}\,\dd\mu
\le
1-F^\ast.
\label{eq:mean_bar_phi_bounds}
\end{equation}
\end{lemma}

\begin{proof}
    As $\widehat{\psi}^{n+1}$ is centered, it holds
    $F(\widehat\xi^{n+\frac23})
=
1-
\int \widehat\varphi^{n+1}\,\dd\mu$. The monotonicity of $F$ along the iterates in \cref{eq:monotonicity} and the invariance of $F$ by translations in $\Acal$, yields $$F^\ast  \le 1-\int  \widehat\varphi^{n+1}\,\dd\mu \le F(\xi^0),$$ which corresponds to the first display. The proof of the second display is similar.
\end{proof}

We can now prove that our normalized iterates are uniformly bounded for the $L^1$ norm.

\begin{proposition}
\label{prop:L1}
Assume that $2\tau D_{\nu}^2 \leq 1$.
Let  $C$ and $\theta$ be as in \cref{lem:coercivity}, set
$ C_{\mathrm{av}}=\max\{|1-F(\xi^0)|,\vert 1-\inf_{\Hcal^\infty} F\vert\}$, and define
$$C_h=2\frac{F(\xi^0)+ C}{\theta},
\; \; 
C_\varphi=2D_\mu C_h+2M_c+C_{\mathrm{av}},
\; \;
U_\psi=F(\xi^0)-1+D_{\nu} C_h+M_c, 
\; \; 
C_\psi=2U_\psi.$$ 
Then, for all $n \ge 0$, it holds
\begin{equation}
\|\widehat h^n\|_{L^1(\mu;\R^d)}
\le
C_h, \qquad \|\widehat\varphi^{n+1}\|_{L^1(\mu)}
\le
C_\varphi, \qquad \|\widehat\psi^{n+1}\|_{L^1(\nu)} \le C_\psi,
\end{equation}
and $$\widehat\psi^{n+1}(y)
\le U_\psi, \qquad \text{for} \quad \nu\text{-a.e. }  y \in \R^d.$$
\end{proposition}

\begin{proof}

Since $2\tau D_{\nu}^2 \leq 1$, by \cref{lem:descent} we have $F(\widehat{\xi}^n) = F({\xi}^n)\leq F({\xi}^0)$. Thus, \cref{lem:coercivity} yields
\begin{equation*}
    \inf_{a\in\R^d}
\|\widehat h^n-a\|_{L^1(\mu)}
\le C_h/2.
\end{equation*}
 Since $\widehat h^n$ is centered, for $a\in \R^d$ we have
\[
|a|
=
\left|
\int (a-\widehat h^n)\,\dd\mu
\right|
\le
\|\widehat h^n-a\|_{L^1(\mu)},
\]
from which follows that
\begin{equation*}
    \|\widehat h^n\|_{L^1(\mu;\R^d)} \le |a| +  \|\widehat h^n - a\|_{L^1(\mu;\R^d)}
\le  2 \|\widehat h^n - a\|_{L^1(\mu;\R^d)}.
\end{equation*}
Since this is true for any $a\in \R^d$ we get
\begin{equation*}
    \|\widehat h^n\|_{L^1(\mu;\R^d)} \leq C_h
\end{equation*}

Applying Jensen's inequality to \cref{eq:hat_varphi} and using \cref{lem:control_integrals_ph} yields
\begin{equation*}
\begin{aligned}
    \widehat\varphi^{n+1}(x)
&\le
-
\int \widehat\psi^{n+1}\,\dd\nu
-
\widehat h^n(x)\cdot(m_\nu-x)+\int c(x,y) \dd\nu(y)
\le 
D_\mu\Vert\widehat h^n(x)\Vert +M_c,
\end{aligned}
\end{equation*}
which proves $
\int (\widehat\varphi^{n+1})_+\,\dd\mu
\le
D_\mu C_h+M_c$. 
Then,
\[
\begin{aligned}
\|\widehat\varphi^{n+1}\|_{L^1(\mu)}
&=
2\int (\widehat\varphi^{n+1})_+\,\dd\mu
-
\int \widehat\varphi^{n+1}\,\dd\mu 
\le
2D_\mu C_h+2M_c+C_{\mathrm{av}} = C_\varphi.
\end{aligned}
\]
Similarly, applying Jensen's inequality to \cref{eq:hat_psi}, by \cref{lem:control_integrals_ph}, we obtain
\[
\begin{aligned}
\widehat\psi^{n+1}(y)
&\le
-\int \tilde\varphi^n\,\dd\mu
+D_{\nu} C_h + M_c
\le F(\xi^0)-1+D_{\nu} C_h + M_c= U_\psi.
\end{aligned}
\]
Since $\widehat\psi^{n+1}$ is centered,
\begin{equation*}
\int  |\widehat\psi^{n+1}| \, \dd \nu = \int  2\widehat\psi^{n+1}_+ - \widehat\psi^{n+1} \, \dd \nu = 2 \int  \widehat\psi^{n+1}_+\le 2 U_\psi = C_\psi,
\end{equation*}
which concludes.
\end{proof}

We now upgrade the uniform $L^1$ bounds of our iterates to uniform $L^\infty$ bounds. To do so we rely on \cref{prop:pullback} and a careful analysis of the gradient step. The gradient is exactly the martingale residual and the argument will hinge on it going to $0$ fast enough. Thus we first prove that the martingale residuals are square-summable.

\begin{lemma}\label{lem:summability}
Assume that $2\tau D_{\nu}^2 \leq 1$, then
\begin{equation}
\sum_{n=0}^{\infty}
\|R^n\|_{L^2(\mu;\R^d)}^2
\le
\frac{2}{\tau}\bigl(F(\xi^0)-F^*\bigr)
<\infty.
\label{eq:R_square_summability}
\end{equation}
Setting $\overline R^n:=\int R^n\,\dd\mu$, we get
\begin{equation}
\sum_{n=0}^{\infty}|\overline R^n|^2 \le \frac{2}{\tau}\bigl(F(\xi^0)-F^*\bigr) <\infty.
\label{eq:mean_R_square_summability}
\end{equation}
\end{lemma}

\begin{proof}
    By \cref{lem:descent}, for all $n\ge 0$, it holds \[
\frac{\tau}{2}
\|R^n\|_{L^2(\mu;\R^d)}^2
\le
F(\xi^n)-F(\xi^{n+1}).
\] Summing the previous display over $n$ telescopes and yields
\cref{eq:R_square_summability}. Jensen's inequality yields
$|\overline R^n|^2\le\|R^n\|_{L^2(\mu;\R^d)}^2$, proving the second inequality.
\end{proof}

We are now in a position to prove our uniform $L^\infty$ bounds.

\begin{proposition}\label{prop:Linfty}
    Assume that $2\tau D_{\nu}^2 \leq 1$.
    Let $\rho$ be as in \cref{lem:upper-bound-zpsih} and  $C(a,b,r)$ as in \cref{prop:pullback}. Set  $C_{\mathrm{pb}} := C(C_\psi+M_c,C_\psi+ M_c,\rho)$, $M_{\mathrm{pb}} := \rho/4$, and fix $C'\in(0,M_{\mathrm{pb}})$. Define 
\begin{equation*}
M_0
=
\max\left\{
C_{\mathrm{pb}},
\frac{2\tau D_{\nu}^2}{M_{\mathrm{pb}}-C'}
\right\}, \quad M_h
=
\max\bigl\{
\|\widehat h^0\|_{L^\infty(\mu;\R^d)},\;
M_0+2\tau D_{\nu}
\bigr\}
+
\frac{4D_{\nu}\bigl(F(\xi^0)-F^*\bigr)}{(C')^2},
\end{equation*}
and set 
$$M_\psi=2M_hD_{\nu}+2M_c,
\qquad
M_\varphi=\max\{M_\psi+M_hD_{\nu}+M_c,\,M_hD_\mu+M_c\}.$$
Then, for all $n \ge 0$, it holds 
$$\|\widehat h^n\|_{L^\infty}\le M_h,
\qquad \|\widehat\psi^{n+1}\|_{L^\infty}\le M_\psi, \qquad \text{and }  \qquad \|\widehat\varphi^{n+1}\|_{L^\infty}\le M_\varphi.$$
\end{proposition}

\begin{proof}
    We prove the three inequalities from the left to the right. Since
\[
h^{n+1}=h^n-\tau R^n \quad \text{and} \quad a^{n+1}=a^n-\tau\overline R^n,
\]
it holds
\begin{equation}
\widehat h^{n+1}
=
\widehat h^n
-
\tau\bigl(R^n-\overline R^n\bigr).
\label{eq:centered_h_update}
\end{equation}
 We will show by induction that
\begin{equation}\label{eq:induction_statement}
\|\widehat h^n\|_{L^\infty(\mu;\R^d)}
\le
G+2\tau D_{\nu}\,b_n
\qquad\text{for every }n\ge0,
\end{equation}
where 
$b_n:=\#\{k<n:\vert\overline R^k\vert>C' \}$ and
$
G:=\max\bigl\{\|\widehat h^0\|_{L^\infty(\mu;\R^d)},\;
M_0+2\tau D_{\nu}\bigr\}$.
For $n=0$ the inequality holds by definition of $G$. Let $n \in \mathbb{N}$ be such that the inequality holds. Let  $x \in X$.
\paragraph{First case, $ \Vert \widehat h^n(x) \Vert \leq M_0$ or $b_{n+1} = b_n+1$.} Since $\Vert R^n-\overline{R}^n\Vert_{L^\infty(\mu;\R^d)} \leq 2 D_{\nu}$, \cref{eq:centered_h_update} implies 
\begin{equation*}
    \Vert\widehat h^{n+1}(x)\Vert\leq \Vert \widehat h^n(x) \Vert+2\tau D_{\nu}.
\end{equation*}
If $\Vert \widehat h^n(x) \Vert \leq M_0$ then
\begin{equation*}
    \Vert\widehat h^{n+1}(x)\Vert \leq M_0 +2\tau D_{\nu} \leq G + 2\tau D_{\nu}b_{n+1}
\end{equation*}
because $b_{n+1} \geq 0$. Else $b_{n+1} = b_{n}+1$ and thus by induction
\begin{equation*}
    \Vert\widehat h^{n+1}(x)\Vert\leq G + 2\tau D_{\nu}b_{n} +2\tau D_{\nu} \leq G + 2\tau D_{\nu}b_{n+1}.
\end{equation*}
In both cases $\Vert\widehat h^{n+1}(x)\Vert\leq G + 2\tau D_{\nu}b_{n+1}.$
\paragraph{Second case, $\Vert \widehat h^n(x) \Vert \geq M_0$ and $b_{n+1} = b_n$.}

In particular $\Vert \widehat h^n(x)\Vert \geq C_{\mathrm{pb}}$. Thus by \cref{prop:pullback}, which applies thanks to \cref{rem:PL_ball,prop:L1}, and invariance by $\Acal$ of $F$ we have
\begin{equation*}
   \widehat h^n(x)\cdot R^n(x) = \widehat h^n(x) \cdot \nabla_h F(\widehat \ph^{n+1},\widehat \psi^{n+1},\widehat h_n)(x) \geq M_{\mathrm{pb}}
\Vert \widehat h^n(x)\Vert.
\end{equation*}
Taking the square in \cref{eq:centered_h_update} and expanding it we obtain
\begin{equation}\label{eq:square_expansion}
\begin{aligned}
\Vert\widehat h^{n+1}(x)\Vert^2
&=
\Vert\widehat h^n(x)\Vert^2
-
2\tau\widehat h^n(x)\cdot R^n(x)
+
2\tau\widehat h^n(x)\cdot\overline R^n
+
\tau^2|R^n(x)-\overline R^n|^2
\\
&\le
\Vert\widehat h^n(x)\Vert^2
-
2\tau(M_{\mathrm{pb}}-C')\Vert\widehat h^n(x)\Vert
+
4\tau^2D_{\nu}^2,
\end{aligned}
\end{equation}
where we used $\vert \overline{R}^n\vert \leq C'$ because $b_{n+1}=b_n$ and $|R^n(x)|\le D_{\nu},|\overline R^n|\le D_{\nu}$ which is a consequence of \cref{prop:bounded-gradients}. As $\Vert\widehat h^n(x)\Vert\ge M_0 \ge (M_{\mathrm{pb}}-C')^{-1}2\tau D_{\nu}^2$, we get 
$$-
2\tau(M_{\mathrm{pb}}-C')\Vert\widehat h^n(x)\Vert
+
4\tau^2D_{\nu}^2 \le 0.$$ By \cref{eq:square_expansion}, this yields $\Vert\widehat h^{n+1}(x)\Vert
\le
\Vert\widehat h^n(x)\Vert \leq G +2\tau D_{\nu} b_{n+1}$ since $b_n = b_{n+1}$. 

Thus in any case we have
\begin{equation*}
    \Vert\widehat h^{n+1}(x)\Vert
\le
G+2\tau D_{\nu}\,b_{n+1},
\end{equation*}
which shows the result for the $L^\infty(\mu;\R^d)$ norm. It remains to get the upper bound
\begin{equation*}
    b_n \leq \frac{2\bigl(F(\xi^0)-F^*\bigr)}{\tau(C')^2}.
\end{equation*}
By \cref{lem:summability}, 
\[
(C')^2b_n 
\le
\sum_{k\in\mathbb{N}, \vert \overline{R}^k\vert> C'}|\overline R^k|^2
\le
\sum_{k=0}^\infty|\overline R^k|^2 \le \frac{2}{\tau}\bigl(F(\xi^0)-F^*\bigr) <\infty,
\]
which yields 
$$
 b_n \leq \frac{2\bigl(F(\xi^0)-F^*\bigr)}{\tau(C')^2}.
$$
Finally we have the desired conclusion for $\widehat h^n$:
\begin{equation*}
    \|\widehat h^n\|_{L^\infty}\le \max\bigl\{
\|\widehat h^0\|_{L^\infty(\mu;\R^d)},\;
M_0+2\tau D_{\nu}
\bigr\}
+
\frac{4D_{\nu}\bigl(F(\xi^0)-F^*\bigr)}{(C')^2}=:M_h
\end{equation*}

We now derive uniform bounds for the normalized potential updates, starting with $\psi$.
Recall
\[
e^{-\widehat\psi^{n+1}(y)}
=
\int 
\exp\!\left(
\tilde\varphi^n(x)
+
\widehat h^n(x)\cdot(y-x)-c(x,y)
\right)\dd \mu(x),
\]
For $y,z\in Y$ and $x \in X$, the two terms in the exponential differ by 
\[
\bigl[\tilde\varphi^n(x)+\widehat h^n(x)\cdot(y-x)-c(x,y)\bigr]
-
\bigl[\tilde\varphi^n(x)+\widehat h^n(x)\cdot(z-x)-c(x,z)\bigr]
=
\widehat h^n(x)\cdot(y-z)-c(x,y)+c(x,z).
\]
Since 
$
\left|
\widehat h^n(x)\cdot(y-z)-c(x,y)+c(x,z)
\right|
\le
M_h|y-z|+2M_c,
$
integrating the corresponding exponential inequalities yields
\[
e^{-M_h|y-z|-2M_c}e^{\widehat\psi^{n+1}(z)}
\le
e^{\widehat\psi^{n+1}(y)}
\le
e^{M_h|y-z|+2M_c}e^{\widehat\psi^{n+1}(z)}.
\]
By taking the logarithm,
\[
\left|
\widehat\psi^{n+1}(y)
-
\widehat\psi^{n+1}(z)
\right|
\le
M_h|y-z|+2M_c
\le
2M_hD_{\nu}+2M_c = M_\psi.
\]
Since \(\widehat\psi^{n+1}\) is centered, \(|\widehat\psi^{n+1}|\) is bounded by its
oscillation, which proves the uniform bound on the normalized $\psi$-iterates. Regarding the normalized $\varphi$-iterates, we have
\[
-\widehat\varphi^{n+1}(x) = \log \int 
\exp\!\left(
\widehat\psi^{n+1}(y)
+
\widehat h^n(x)\cdot(y-x)-c(x,y)
\right)\dd \nu(y)
\le M_\psi+M_hD_{\nu}+M_c,
\]
and therefore
$
\widehat\varphi^{n+1}(x)
\ge
-M_\psi-M_hD_{\nu}-M_c.
$
For the reverse inequality, Jensen's inequality and the centering of
\(\widehat\psi^{n+1}\) give
\[
\begin{aligned}
\widehat\varphi^{n+1}(x)
&\le
-
\int 
\left(
\widehat\psi^{n+1}(y)
+
\widehat h^n(x)\cdot(y-x)-c(x,y)
\right)\dd \nu(y)
\le
M_hD_\mu+M_c,
\end{aligned}
\]
which concludes the bounds for the normalized iterates.
\end{proof}

We now conclude this subsection by proving that our iterates are contained in a set $K_M$, where $M$ depends on the radius of the $L^\infty$  balls containing the renormalized iterates.

\begin{lemma}\label{lem:segments}
Let $M_\varphi$, $M_\psi$ and $M_h$ be defined as in
\cref{prop:Linfty}. Set
\[
A:=M_\varphi+M_\psi+M_hD_{\nu}+M_c, \quad \text{and} \quad
B:=\max\{4||u_{\xi^0}||_{L^{\infty}(\mu\otimes\nu)},A\}+\tau D_{\nu}^2.
\]
Then, for $M_u:=2B$ and every $n\geq 0$,
$
\xi^n, \xi^{n+\frac13}, 
\xi^{n+\frac23}\in K_{M_u}.
$
\end{lemma}

\begin{proof}
By \cref{prop:Linfty} and the invariance
of $u_\xi$ under the affine normalization, we have,
for every $n\geq 1$,
$
\|u_{\xi^{n+\frac23}}\|_{L^\infty(\mu\otimes\nu)}= \|u_{\widehat\xi^{n+\frac23}}\|_{L^\infty(\mu\otimes\nu)}
\leq A.
$
Moreover, the gradient update gives
\[
u_{\xi^{n+1}}(x,y)
=
u_{\xi^{n+\frac23}}(x,y)
-\tau R^n(x)\cdot(y-x).
\]
Since $\|R^n\|_{L^\infty(\mu)}\leq D_{\nu}$
by \cref{prop:bounded-gradients},
we obtain for $n \geq 1$
\begin{equation} \label{eq:bound_u_xi}
\|u_{\xi^{n+1}}\|_{L^\infty(\mu\otimes\nu)}
\leq
\|u_{\xi^{n+\frac23}}\|_{L^\infty(\mu\otimes\nu)}
+\tau D_{\nu}^2 \leq A+\tau D_{\nu}^2. 
\end{equation}
To control the first iteration, observe that
$\|u_{\xi^0}\|_{L^\infty(\mu\otimes\nu)}<\infty$,
since $\xi^0\in\mathcal H^\infty$.
By the two Sinkhorn updates,
\begin{equation}\label{eq:bound_u_xi_n0}
\begin{aligned}
u_{\xi^{1/3}}(x,y)
=u_{\xi^0}(x,y)
-\log\int  e^{u_{\xi^0}(z,y)}\dd\mu(z)
&\implies
\|u_{\xi^{1/3}}\|_{L^\infty(\mu\otimes\nu)}\leq 2\|u_{\xi^0}\|_{L^\infty(\mu\otimes\nu)},
\\
u_{\xi^{2/3}}(x,y)
=u_{\xi^{1/3}}(x,y)
-\log\int  e^{u_{\xi^{1/3}}(x,z)}\dd\nu(z)
&\implies
\|u_{\xi^{2/3}}\|_{L^\infty(\mu\otimes\nu)}\leq 4\|u_{\xi^0}\|_{L^\infty(\mu\otimes\nu)}.
\end{aligned}
\end{equation}
The gradient update then yields
\[
\|u_{\xi^1}\|_{L^\infty(\mu\otimes\nu)}
\leq 4\|u_{\xi^0}\|_{L^\infty(\mu\otimes\nu)}+\tau D_{\nu}^2.
\]
Combining estimates in \cref{eq:bound_u_xi_n0} with \cref{eq:bound_u_xi}
we conclude that
$
\|u_{\xi^n}\|_{L^\infty(\mu\otimes\nu)}\leq B,
$
and that
$
\|u_{\xi^{n+2/3}}\|_{L^\infty(\mu\otimes\nu)}\leq B
$
for
$
n\geq0.
$
It remains to control $\xi^{n+\frac13}$ for $n\geq0$.
By the $\psi$-update, we similarly obtain
\[
u_{\xi^{n+\frac13}}(x,y)
=
u_{\xi^n}(x,y)
-\log\int  e^{u_{\xi^n}(z,y)}\dd\mu(z)
\implies
\|u_{\xi^{n+\frac13}}\|_{L^\infty(\mu\otimes\nu)}
\leq M_u.
\]
\end{proof}

\subsection{Proof of Convergence}\label{sec:proof_of_convergence}
We have shown in \cref{sec:bounded-iterates} that $\xi^n \in K_{M_u}$. This fact in conjunction with \cref{sec:properties} which ensures that $F$ satisfies all the conditions of \cref{thm:abstract}, allows us to apply it to show convergence of SISTA.  Throughout, we assume \cref{ass:qualification} holds, $\xi^0\in\Hcal^\infty$, and $2\tau D_{\nu}^2\leq 1$. As SISTA updates $\psi$ then $\varphi$, we take $N=3$ and
\[
\Hcal_1:=L^2(\nu),\qquad\Hcal_2:=L^2(\mu),\qquad\Hcal_3:=L^2(\mu;\R^d),
\]
so that $\Hcal=\Hcal_2\times\Hcal_1\times\Hcal_3$ is the space of \cref{sec:setting} with its first two components swapped, with $\xi=(\psi,\varphi,h)$, and $J_1\psi=(\psi,0,0)$, $J_2\varphi=(0,\varphi,0)$, $J_3h=(0,0,h)$. The function $F:\Hcal\to(-\infty,+\infty]$ is the dual functional \cref{eq:main_dual}. The $\psi$- and $\varphi$-updates of \cref{alg:SISTA} are exact steps and the $h$-update is a gradient step with step size $\tau$. Formally, $\mathcal{E}=\{1,2\}$ and $\mathcal G=\{3\}$. We take $K:=K_{M_u}$, where $K_M$ is defined in \cref{eq:KM} and $M_u$ as given by \cref{lem:segments}. We define 
$\nabla F(\xi)=(\nabla_\psi F(\xi),\nabla_\varphi F(\xi),\nabla_hF(\xi))\in\Hcal$, where we defined
\begin{equation}\label{eq:gradient_def}
    \begin{cases}
    \nabla_\psi F(\xi):=\int e^{u_\xi(x,\cdot)}\dd \mu(x)-1\\
    \nabla_\varphi F(\xi):=\int e^{u_\xi(\cdot,y)}\dd \nu(y)-1\\
    \nabla_hF(\xi):=\int (y-\cdot)\,e^{u_\xi(\cdot,y)}\dd \nu(y)
\end{cases}.
\end{equation}
We start by checking that the \cref{point:finite_min}--\cref{point:strong_block_convexity} are satisfied.

\begin{lemma}\label{lem:SISTA-A}
Define
\[
L:=3e^{M_u}\max\{1,D_{\nu}^2\},\qquad\kappa:=\min\left\{
\frac12,
\frac{\lambda_\nu}{1+2D_\mu^2}
\right\},\qquad\sigma:=\kappa e^{-M_u},\qquad m:=e^{-M_u}.
\]
Then $K$ is a closed convex subset of $\Hcal$ on which $F$ is finite, and properties \cref{point:finite_min}--\cref{point:strong_block_convexity} hold.
\end{lemma}

\begin{proof}
By \cref{lem:prop_KM}, $K$ is closed, convex, and $F$ is finite on it. By \cref{lem:segments} it contains the iterates. \Cref{point:finite_min} holds by finiteness of the problem, while  \cref{point:smooth} is \cref{prop:smoothness-global}. We now show \cref{point:convex}. Let $\xi=(\psi,\varphi,h)\in K$, $\xi'=(\psi',\varphi',h')\in\mathcal H$, and set $(\delta\psi,\delta\varphi,\delta h):=\xi'-\xi$. By convexity of the exponential, $e^{u_{\xi'}}\ge e^{u_\xi}+e^{u_\xi}(u_{\xi'}-u_\xi)$ pointwise. By \cref{eq:gradient_def}, integration against $\mu \otimes \nu$ and applying Fubini yields 
\begin{align*}
    \int e^{u_{\xi'}}\dd(\mu\otimes\nu) \ge \int &e^{u_{\xi}}\dd(\mu\otimes\nu) 
+\int\big(\nabla_\varphi F(\xi)+1\big)\,\delta\varphi\dd\mu
\\ &+\int\big(\nabla_\psi F(\xi)+1\big)\,\delta\psi\dd\nu
+\int\nabla_h F(\xi)\cdot\delta h\dd\mu,
\end{align*}
which by subtracting $\int \varphi' d \mu + \int \psi' d \nu$ rewrites
$F(\xi')-F(\xi) \ge \langle \nabla F(\xi),\xi'-\xi \rangle$ and shows \cref{point:convex}. Finally, \cref{point:PLI} is \cref{prop:pli}, and \cref{point:strong_block_convexity} is \cref{prop:block-strong-convexity}.
\end{proof}

We now check that Assumptions \cref{point:belongs}--\cref{point:lipschitz_on_segments} hold.

\begin{lemma}\label{lem:SISTA-I}
Let $n\ge0$.
\begin{enumerate}[label=(\roman*), ref=(\roman*)]
\item\label{SI:belongs} $\xi^{n+\frac k3}\in K$ for $k\in\{0,1,2\}$, which is \cref{point:belongs};
\item\label{SI:steps} For all $k\in\{1,2,3\}$, $\xi^{n+\frac {k}{3}}-\xi^{n+\frac {k-1}{3}}\in J_k(\Hcal_k).$ Moreover, $$\nabla_1F(\xi^{n+\frac13})=0,\quad \nabla_2F(\xi^{n+\frac23})=0, \quad \text{and }\quad\xi^{n+1}=\xi^{n+\frac23}-\tau J_3\nabla_3F(\xi^{n+\frac23});$$
\item\label{SI:lip} for $s\in[0,1]$, writing $\xi_s^{n+1}:=\xi^{n+\frac23}+s(\xi^{n+1}-\xi^{n+\frac23})$,
\[
\lVert\nabla_3F(\xi^{n+1}_s)-\nabla_3F(\xi^{n+\frac23})\rVert\le\tau^{-1}\lVert\xi^{n+1}_s-\xi^{n+\frac23}\rVert ,
\]
which is \cref{point:lipschitz_on_segments}.
\end{enumerate}
\end{lemma}

\begin{proof}
By \cref{prop:well-defined}, all the partial iterates belong to $\Hcal^\infty$.

\cref{SI:belongs} This is \cref{lem:segments}.

\cref{SI:steps} This follows directly by definition of the SISTA.

\cref{SI:lip}  As $\varphi^{n+1}=Z_{\psi^{n+1},h^n}$, defining $Z_{\psi,h}$ as in \cref{eq:normalization-constant}, it holds  $$\xi^{n+\frac23}=(\psi^{n+1},Z_{\psi^{n+1},h^n},h^n) \quad \text{and}\quad \xi^{n+1}_s=(\psi^{n+1} Z_{\psi^{n+1},h^n},h^n-s\tau R^n),$$ with $R^n=\nabla_hF(\xi^{n+\frac23})$. By \cref{prop:bounded-gradients}, $\lVert R^n\rVert_{L^\infty(\mu;\R^d)}\le D_{\nu}$. Thus \Cref{prop:lipschitz-gradients} yields
$$
\lVert\nabla_3F(\xi_s)-\nabla_3F(\xi^{n+\frac23})\rVert\le 
D_{\nu}^2 e^{\Vert s\tau R^n\Vert_{L^\infty(\mu) }D_{\nu}} \Vert s\tau R^n\Vert_{L^2(\mu;\R^d)}
\le
D_{\nu}^2e^{\tau D_{\nu}^2}\lVert\xi^{n+1}_s-\xi^{n+\frac23}\rVert,
$$
where we used $\lVert\xi^{n+1}_s-\xi^{n+\frac23}\rVert=\lVert s\tau R^n\rVert_{L^2(\mu;\R^d)}$. To conclude use that $\tau D_{\nu}^2 \leq \frac{1}{2}$ and $x\leq \frac{1}{2}$ implies $xe^x \leq 1$, which in turn shows $D_{\nu}^2e^{\tau D_{\nu}^2} \leq \tau^{-1}$.
\end{proof}

We are now in position to prove the main result \cref{thm:main-result}.

\begin{proof}[Proof of \cref{thm:main-result}]
    By \cref{lem:SISTA-A} and \cref{lem:SISTA-I}, \cref{thm:abstract} applies. Moreover, as $K_M \subset \mathcal H^\infty$ by \cref{lem:prop_KM}, $\xi^\infty \in \mathcal H^\infty $.
\end{proof}
We end this section with a corollary stating that the iterates are bounded in $\Hcal^\infty$ which motivates the use of the algorithm without normalization in the numerical section.

\begin{corollary}\label{cor:boundedness}
    Under the assumptions and the notations of \cref{thm:main-result} we have
    \begin{equation*}
        \max\{\Vert \xi^\infty\Vert_{\Hcal^\infty},\sup_{n\geq 1} \Vert \xi^n \Vert_{\Hcal^\infty} \}\leq M_\ph + M_\psi + M_h + (3+2D_{\nu})\left(1+4 \dfrac{\sqrt{2E_0/\alpha}}{1-\sqrt q} +2 \Vert \xi^0 \Vert_{\Hcal}\right).
    \end{equation*}
    where $M_\ph,M_\psi$ and $M_h$ are defined in \cref{prop:Linfty}.
\end{corollary}
\begin{proof}
    By \cref{thm:main-result} we have a uniform control over the $\Hcal$ norms
    \begin{equation}\label{eq:cor_unif_h_bound}
        \Vert \xi^n \Vert_{\Hcal} \leq \Vert \xi^n - \xi^\infty \Vert_{\Hcal} + \Vert \xi^\infty - \xi^0 \Vert_{\Hcal} + \Vert \xi^0 \Vert_{\Hcal} \leq 2 \dfrac{\sqrt{2E_0/\alpha}}{1-\sqrt q} + \Vert \xi^0 \Vert_{\Hcal}.
    \end{equation}
Given the definition of the normalized iterates in \cref{eq:def_nomalized}, we have
\begin{equation}\label{eq:cor_bounded_iterates}
\left\{
\begin{aligned}
&\Vert h^{n+1}\Vert_{L^\infty(\mu;\R^d)}
 \leq \Vert \widehat h^{n+1}\Vert_{L^\infty(\mu;\R^d)} + \Vert h^{n+1} \Vert_{L^1(\mu;\R^d)}\\
&\Vert\psi^{n+1}\Vert_{L^\infty(\nu)}
\leq\Vert\widehat\psi^{n+1}\Vert_{L^\infty(\nu)} +\Vert h^n \Vert_{L^1(\mu;\R^d)}D_{\nu}+\Vert \psi^{n+1}\Vert_{L^1(\nu)}\\
&\Vert\varphi^{n+1}\Vert_{L^\infty(\mu)}
\leq \Vert\widehat\varphi^{n+1}\Vert_{L^\infty(\mu)}+\Vert h^n \Vert_{L^1(\mu;\R^d)}D_{\nu}+\Vert \psi^{n+1}\Vert_{L^1(\nu)}
\end{aligned}
\right. .
\end{equation}
Since $\mu$ and $\nu$ are probability measures, the $L^1$ norms are bounded by the corresponding $L^2$ norms. Summing the inequalities in \cref{eq:cor_bounded_iterates} and using
\cref{prop:Linfty}, we obtain
\begin{equation*}
\begin{aligned}
\|\xi^{n+1}\|_{\Hcal^\infty}
&\leq M_\ph+M_\psi+M_h
+\|h^{n+1}\|_{L^2(\mu;\R^d)}\\
&\quad+2D_{\nu}\|h^n\|_{L^2(\mu;\R^d)}
+2\|\psi^{n+1}\|_{L^2(\nu)}\\
&\leq M_\ph+M_\psi+M_h
+3\|\xi^{n+1}\|_{\Hcal}
+2D_{\nu}\|\xi^n\|_{\Hcal}.
\end{aligned}
\end{equation*}
  Combining this with \cref{eq:cor_unif_h_bound} gives the claimed inequality. The same bound holds for $\xi^\infty$, since $\xi^n\to\xi^\infty$ in $\Hcal$ and bounded balls of $\Hcal^\infty$ are closed in $\Hcal$.
\end{proof}

\section{Numerical Experiments}
\label{sec:numerics}
In this section, we investigate the numerical behavior of SISTA through three experiments. The first considers the distance cost $c(x,y)=|y-x|$ and illustrates the linear convergence predicted by our main theorem. The second computes a two-dimensional martingale Schrödinger bridge $(c=0)$ and constructs the associated continuous-time Föllmer martingale. The third examines the role of irreducibility and non-degeneracy, illustrating the growth of the martingale multiplier when irreducibility fails and the deterioration of the observed convergence rate near a degenerate limit.

We first introduce the finite-dimensional formulation used in the experiments. 
Let
\[
\mu_N=\sum_{i=1}^{N_\mu}\mu_i\delta_{x_i},
\qquad
\nu_N=\sum_{j=1}^{N_\nu}\nu_j\delta_{y_j},
\]
be discrete approximations of $\mu$ and $\nu$, with positive weights summing to one. Setting $c_{ij}=c(x_i,y_j)$, the corresponding dual functional is
\[
F_N(\varphi,\psi,h)
=\sum_{i=1}^{N_\mu}\sum_{j=1}^{N_\nu}
\mu_i\nu_j
\exp\bigl(\varphi_i+\psi_j+h_i\cdot(y_j-x_i)-c_{ij}\bigr)
-\sum_{i=1}^{N_\mu}\mu_i\varphi_i
-\sum_{j=1}^{N_\nu}\nu_j\psi_j.
\]
The discrete SISTA iteration is given in \cref{alg:discrete-sista}.

\begin{algorithm}[H]
\caption{Discrete SISTA scheme}
\label{alg:discrete-sista}
\begin{algorithmic}[1]

\State \textbf{Input:}
$\mu_N$, $\nu_N$, cost matrix $(c_{ij})$,
step size $\tau>0$, and
$\xi^0=(\varphi^0,\psi^0,h^0)$.

\State \textbf{Repeat until convergence:}

\Statex
\State \textbf{(S)} \textit{Sinkhorn block:}

\Statex
\[
\psi_j^{n+1}
\leftarrow
-\log\left(
\sum_{i=1}^{N_\mu}
\mu_i
\exp\!\left(
\varphi_i^n+h_i^n\cdot(y_j-x_i)-c_{ij}
\right)
\right),
\qquad
j=1,\ldots,N_\nu.
\]

\Statex
\[
\varphi_i^{n+1}
\leftarrow
-\log\left(
\sum_{j=1}^{N_\nu}
\nu_j
\exp\!\left(
\psi_j^{n+1}
+h_i^n\cdot(y_j-x_i)-c_{ij}
\right)
\right),
\qquad
i=1,\ldots,N_\mu.
\]

\Statex
\State \textbf{(G)} \textit{Gradient block:}

\Statex
\[
R_i^n
:=
\sum_{j=1}^{N_\nu}
\nu_j(y_j-x_i)
\exp\!\left(
\varphi_i^{n+1}+\psi_j^{n+1}
+h_i^n\cdot(y_j-x_i)-c_{ij}
\right).
\]

\Statex
\[
h_i^{n+1}
\leftarrow
h_i^n-\tau R_i^n,
\qquad
i=1,\ldots,N_\mu.
\]

\end{algorithmic}
\end{algorithm}

\subsection{SISTA for the Distance Cost}

We consider the one-dimensional distance cost $c_\varepsilon(x,y)=|y-x|/\varepsilon$, with $\mu=\operatorname{Unif}[-1,1]$ and $\nu=\operatorname{Unif}[-2,2]$. This example is motivated by mathematical finance, where the corresponding unregularized martingale optimal transport problem arises in the model-independent pricing of forward-starting straddles~\cite{HoKl15}.  Since \cref{ass:qualification} is satisfied, the linear convergence result of \cref{thm:main-result} applies. For the numerical implementation, we use the midpoint discretizations
\[
\mu_N=\frac1N\sum_{i=1}^N\delta_{-1+(2i-1)/N},
\qquad
\nu_N=\frac1{2N}\sum_{j=1}^{2N}\delta_{-2+(2j-1)/N}.
\]
For even $N$, these discretizations preserve convex order and satisfy $\operatorname{supp}(\mu_N)\subset\operatorname{supp}(\nu_N)$. We first fix $\varepsilon=1$ and run SISTA with step size $\tau=0.02$ and initialization $\xi^0=(0,0,0)$.

\Cref{fig:straddle} displays the martingale residual $||R^n||_{L^2(\mu_N)}$, the multiplier error $||h^n-h_N^{\mathrm{ref}}||_{L^2(\mu_N)}$, and the dual objective gap $F_N(\xi^n)-F_N^{\mathrm{ref}}$. The reference values are obtained by running SISTA for $2000$ iterations. All three quantities exhibit approximately geometric decay, consistently with \cref{thm:main-result}. In particular, the objective gap decreases approximately twice as fast as the martingale residual and multiplier error on the logarithmic scale. Moreover, the convergence curves nearly overlap across discretization levels, suggesting little deterioration of the observed rate under mesh refinement.

\begin{figure}[H]
    \centering
    \includegraphics[width=\textwidth]{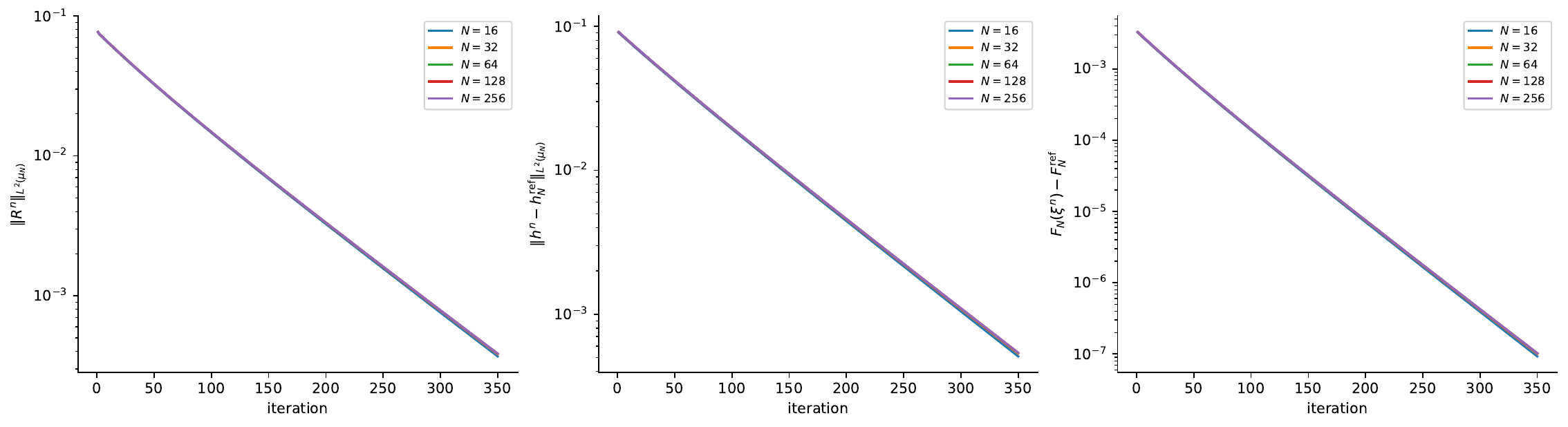}
    \caption{Convergence of SISTA for the distance cost with
$\varepsilon=1$ and different discretization levels $N$.
The plots display the decay of the martingale residual
$\|R^n\|_{L^2(\mu_N)}$, the multiplier error
$\|h^n-h_N^{\mathrm{ref}}\|_{L^2(\mu_N)}$, and the dual objective gap
$F_N(\xi^n)-F_N^{\mathrm{ref}}$. }
    \label{fig:straddle}
\end{figure}

We next investigate the influence of the regularization parameter $\varepsilon$ on the structure of the computed coupling. For this purpose, we fix $N=100$ and display the results for $\varepsilon\in\{1,\,0.1,\,0.01\}$. To facilitate the computation for small values of $\varepsilon$, we employ an $\varepsilon$-scaling procedure, successively decreasing $\varepsilon$ through the values $\{1,\,0.5,\,0.25,\,0.1,\,0.05,\,0.02,\,0.01\}$. More precisely, given $\varepsilon_2<\varepsilon_1$, if $(\varphi_1,\psi_1,h_1)$ are the dual potentials obtained by running SISTA for $\varepsilon=\varepsilon_1$, we initialize SISTA for $\varepsilon=\varepsilon_2$ with $\varepsilon_1 \varepsilon_2^{-1}(\varphi_1,\psi_1,h_1)$.

In \Cref{fig:straddle-heatmaps}, for each value of $\varepsilon$, we display the density with respect to $\mu_N\otimes\nu_N$ of the approximation of the solution of \labelcref{eq:introEMOT} obtained via SISTA, that is,
\[
\frac{\dd\pi_N^{\varepsilon}}{\dd(\mu_N\otimes\nu_N)}(x_i,y_j)
=
\exp\left(
\varphi_i^{\varepsilon}+\psi_j^{\varepsilon}
+h_i^{\varepsilon}(y_j-x_i)
-\frac{|y_j-x_i|}{\varepsilon}
\right).
\]
As expected, decreasing $\varepsilon$ leads to a more concentrated coupling, reflecting the increasing importance of the distance cost relative to the entropic regularization.

\begin{figure}[H]
    \centering
    \includegraphics[width=\textwidth]{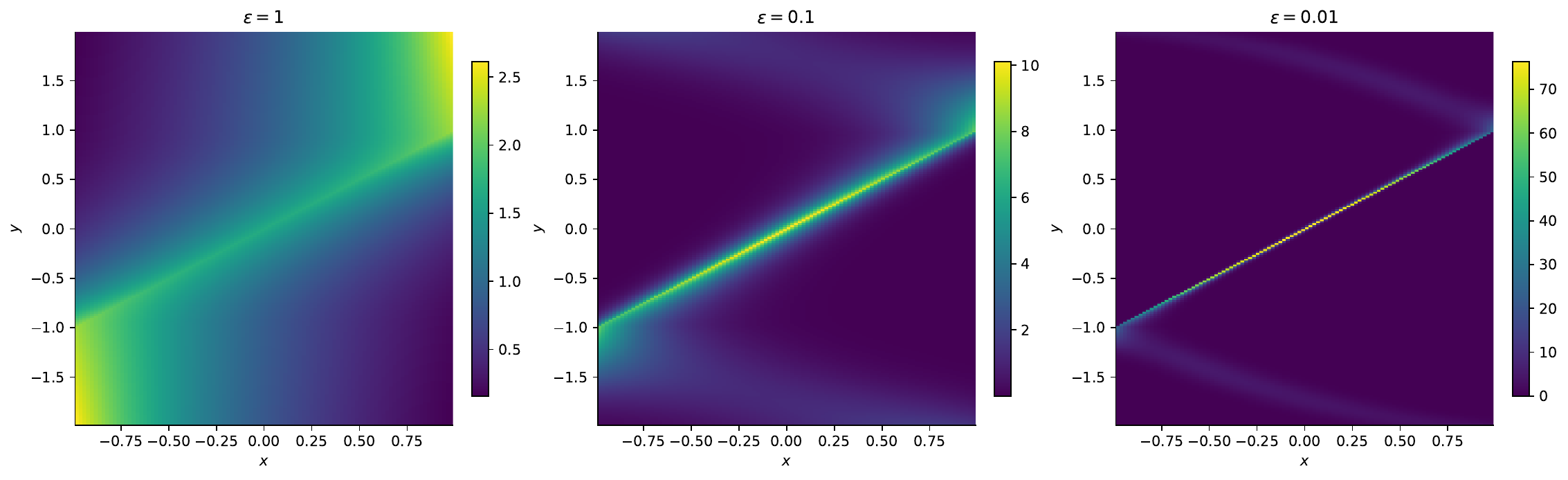}
    \caption{Density of the entropic martingale coupling with respect to
    $\mu_N\otimes\nu_N$ for the distance cost
    $c_\varepsilon(x,y)=|y-x|/\varepsilon$, with $N=100$.}
    \label{fig:straddle-heatmaps}
\end{figure}

\subsection{From the Martingale Schrödinger Bridge to the Föllmer Martingale}

We compute the martingale Schrödinger bridge ($c=0$) for the two-dimensional marginals
\[
\mu=\operatorname{Unif}(\mathbb S^1),
\qquad
\nu=
\frac12\operatorname{Unif}\!\left(\frac12\mathbb S^1\right)
+
\frac12\operatorname{Unif}\!\left(\frac85\mathbb S^1\right),
\]
where $\mathbb S^1:=\{x\in\mathbb R^2:\ |x|=1\}$.
We discretize $\mu$ using $64$ equally spaced points on the unit circle and $\nu$ using $64$ points on each of the two circles. We run SISTA with $\tau=0.04$ for $20000$ iterations and use the resulting coupling to approximate the associated continuous-time Föllmer martingale, following~\cite{BaBeBiLe26}. After the final SISTA iteration, we perform an additional $\varphi$-update to enforce the first marginal constraint. The resulting Gibbs plan defines the conditional probabilities
\[
p_{ij}
=\nu_j\exp\bigl(\varphi_i+\psi_j+h_i\cdot(y_j-x_i)\bigr),
\qquad
\sum_jp_{ij}=1.
\]

Conditionally on $\{X=x_i\}$, we sample $Y$ according to $(p_{ij})_j$ and define
\[
\widehat X_t=(1-t)X+tY+\beta_t,
\qquad 0\leq t<1,
\]
where $\beta$ is a standard two-dimensional Brownian bridge independent of $(X,Y)$. The associated filtering martingale is
\[
M_t
=\mathbb E\left[Y\mid X,(\widehat X_r)_{0\leq r\leq t}\right]
=\sum_j y_jw_{ij}(t,\widehat X_t),
\]
on $X=x_i$, where
\[
w_{ij}(t,z)
=\mathbb P(Y=y_j\mid X=x_i,\widehat X_t=z).
\]
Using the Brownian bridge representation $\beta_t=(1-t)B_{t/(1-t)}$, we obtain
\[
\frac{\widehat X_t-X}{1-t}
=\frac{t}{1-t}(Y-X)+B_{t/(1-t)}.
\]
Thus, under the time change $s=t/(1-t)$, the observation process is equivalent to the Brownian filtering model $s(Y-X)+B_s$ in~\cite[Theorem~5.1.]{BaBeBiLe26}.

\Cref{fig:2d_follmer} illustrates sample trajectories of the resulting filtering martingale, providing a numerical approximation of the Föllmer martingale associated with the continuous martingale Schrödinger bridge. These trajectories may be compared with those of the Bass martingale in~\cite[Figure~1]{BaBeSchaTschi26}.

\begin{remark}
Since $\nu$ is singular with respect to the Gaussian reference measure $\gamma$, we have $H(\nu\mid\gamma)=+\infty$. Thus, the finite-energy characterization of~\cite[Theorem~4.12.]{BaBeBiLe26} does not apply. A rigorous justification of our numerical approximation would require suitable stability results, that we leave for further work.
\end{remark}

\subsection{Role of Irreducibility and Non-degeneracy}

We conclude the numerical section by examining the role of the irreducibility and non-degeneracy assumptions in \cref{thm:main-result}. The first example violates irreducibility while preserving non-degeneracy, whereas the second considers a sequence of irreducible, non-degenerate problems approaching a degenerate continuous limit.

\paragraph{Failure of irreducibility.} Consider the one-dimensional marginals
\[
\mu=\frac12(\delta_{-2}+\delta_2),
\qquad
\nu=\frac14(\delta_{-3}+\delta_{-1}+\delta_1+\delta_3).
\]
They are in convex order and satisfy the non-degeneracy condition, but irreducibility fails. Indeed, the unique martingale coupling is
\[
m^\star=
\begin{pmatrix}
\frac14&\frac14&0&0\\
0&0&\frac14&\frac14
\end{pmatrix},
\]
whose conditional laws are not equivalent to $\nu$. Since $m^\star$ has zero entries, it cannot be represented by finite Gibbs potentials. Nevertheless, Gibbs couplings can approximate $m^\star$: convergence of the dual objective to its infimum implies convergence of the associated Gibbs coupling to $m^\star$.
The optimal entropy and dual infimum are explicit:
\[
H(m^\star\mid\mu\otimes\nu)=\log 2,
\qquad
F^\star=1-\log 2.
\]

We run SISTA with step size $\tau=0.01$ for $20000$ iterations, starting from $\xi^0=(0,0,0)$.
\Cref{fig:failure_of_irreducibility} shows that the martingale residual and dual objective gap decrease throughout the computation, while $\|h^n\|_{L^2(\mu)}$ continues to grow. This behavior reflects the absence of finite Gibbs potentials representing $m^\star$: approximating its zero entries requires increasingly large differences between the martingale multipliers. Thus, the experiment illustrates the role of irreducibility in ensuring uniform multiplier bounds.

\begin{figure}[H]
    \centering

    \includegraphics[width=\textwidth]{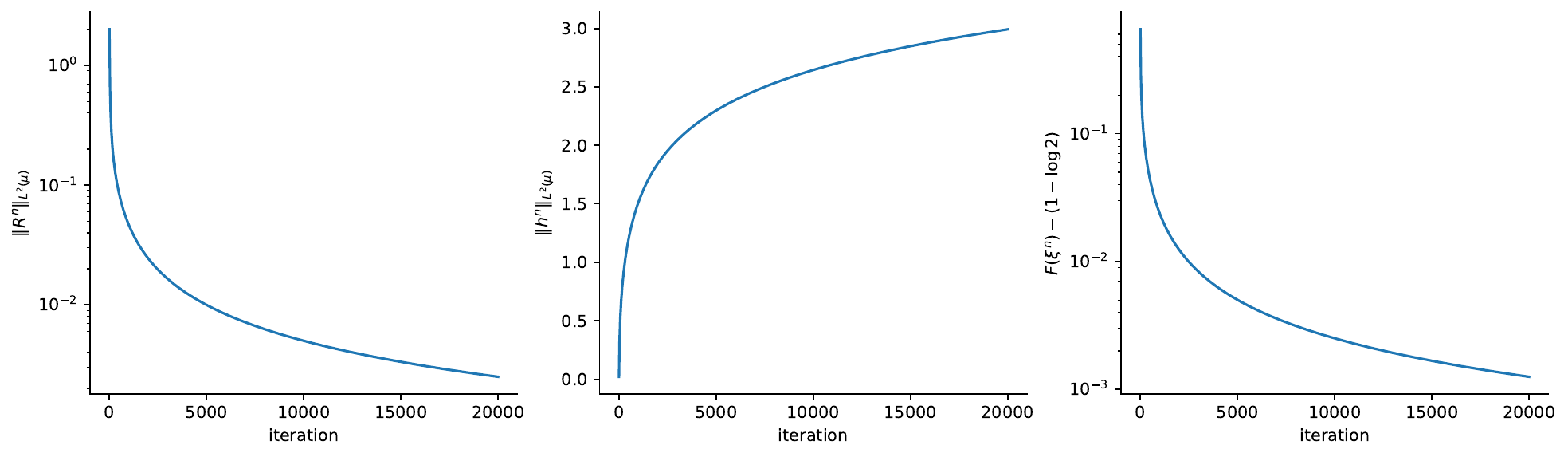}

   \caption{Behavior of SISTA when irreducibility fails. The plots display the martingale residual $\|R^n\|_{L^2(\mu)}$, the multiplier norm $\|h^n\|_{L^2(\mu)}$, and the dual objective gap $F(\xi^n)-F^\star$.
    }
    \label{fig:failure_of_irreducibility}

\end{figure}

\paragraph{Failure of non-degeneracy} We next consider a sequence of discrete problems satisfying the assumptions of \cref{thm:main-result} and approaching a degenerate continuous limit. Let
\[
\dd \mu(x)=(1-|x|)\mathbf 1_{[-1,1]}(x)\dd x,
\qquad
\nu=\operatorname{Unif}[-1,1].
\]
The pair is irreducible but fails the non-degeneracy condition, since
\[
\operatorname{supp}(\mu)=[-1,1]
\not\subset
\operatorname{int}\operatorname{conv}(\operatorname{supp}(\nu))
=(-1,1).
\]
A discrete approximation placing positive mass at an extreme point of the convex hull of $\operatorname{supp}(\nu_N)$ cannot preserve irreducibility, since the martingale constraint forces the corresponding conditional law to be concentrated at that point.

We approximate the continuous marginals using equal-weight midpoint-quantile discretizations. Let $Q_\mu$ and $Q_\nu$ denote their quantile functions and set
\[
u_i=\frac{i-\frac12}{N},
\qquad
x_i=Q_\mu(u_i),\qquad y_i=Q_\nu(u_i),
\quad i=1,\ldots,N,
\]
where
\[
Q_\mu(u)=
\begin{cases}
-1+\sqrt{2u},&0\leq u\leq\frac12,\\
1-\sqrt{2(1-u)},&\frac12<u\leq1,
\end{cases}
\qquad Q_\nu(u)=2u-1.
\]
We define
\[
\mu_N=\frac1N\sum_{i=1}^N\delta_{x_i},
\qquad
\nu_N=\frac1N\sum_{i=1}^N\delta_{y_i}.
\]
These discretizations preserve convex order and converge weakly to $\mu$ and $\nu$, respectively. Moreover,
\[
\max\operatorname{supp}(\mu_N)=1-\frac1{\sqrt N},
\qquad
\max\operatorname{supp}(\nu_N)=1-\frac1N,
\]
so that the non-degeneracy margin satisfies
\[
\delta_N
=\operatorname{dist}\!\left(
\operatorname{supp}(\mu_N),
\partial\operatorname{conv}(\operatorname{supp}(\nu_N))
\right)
=\frac1{\sqrt N}-\frac1N\longrightarrow0.
\]

For each discretization, we run SISTA with step size $\tau=0.10$ for $3000$ iterations, starting from $\xi^0=(0,0,0)$. The reference multiplier $h_N^{\mathrm{ref}}$ and objective value $F_N^{\mathrm{ref}}$ are obtained by continuing SISTA for $12000$ iterations.

\Cref{fig:failure_of_nondegeneracy} displays the martingale residual $\|R^n\|_{L^2(\mu_N)}$, the multiplier error $\|h^n-h_N^{\mathrm{ref}}\|_{L^2(\mu_N)}$, and the dual objective gap $F_N(\xi^n)-F_N^{\mathrm{ref}}$ for $N=16,32,64,128,256$. All three quantities decrease, but their decay becomes progressively slower as $N$ increases. As in \Cref{fig:straddle}, the objective gap generally decreases faster than the residual and multiplier error.

In contrast to \Cref{fig:straddle}, where the convergence curves nearly overlap across discretization levels, \Cref{fig:failure_of_nondegeneracy} shows a marked deterioration of the observed convergence rate as $\delta_N\to0$. Although each finite discretization remains compatible with geometric convergence, the plots do not suggest a rate that is uniform under mesh refinement. This is consistent with the role of non-degeneracy in the estimates of the preceding sections: as the discrete problems approach the degenerate continuous limit, the constants governing boundedness and convergence may deteriorate.
\begin{figure}[H]
    \centering

    \includegraphics[width=\textwidth]{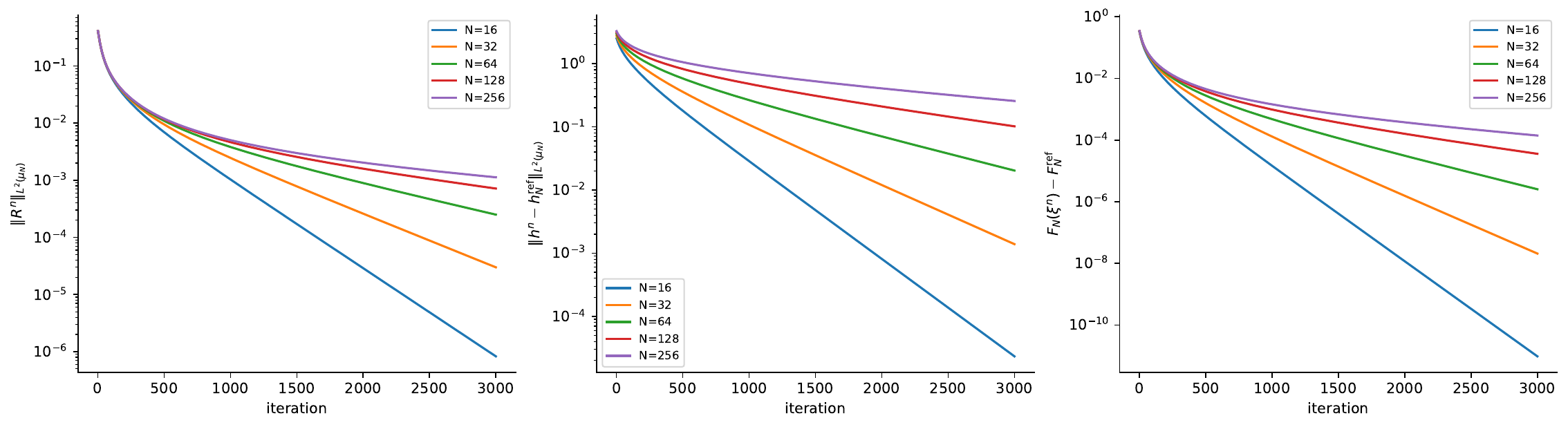}

\caption{Convergence of SISTA as the discrete problems approach a degenerate continuous limit, for $N=16,32,64,128,256$, with non-degeneracy margin $\delta_N=N^{-1/2}-N^{-1}$. The plots display the martingale residual $\|R^n\|_{L^2(\mu_N)}$, the multiplier error $\|h^n-h_N^{\mathrm{ref}}\|_{L^2(\mu_N)}$, and the dual objective gap $F_N(\xi^n)-F_N^{\mathrm{ref}}$.}
    \label{fig:failure_of_nondegeneracy}

\end{figure}

\paragraph{Acknowledgments}
The authors acknowledge the use of ChatGPT (OpenAI) for assistance with language editing, assessment of the references and code development. All mathematical arguments, proofs, and conclusions were developed and independently verified by the authors.

\appendix

\section{Equivalence between Assumptions}\label{app:ass}

In \cite[Remark 2.6]{KaReWa26} the authors show that irreducibility implies strict convex order. Partly reproducing  the proof of \cite[Theorem 2.9.]{CaMaSy26} we prove that, under non-degeneracy, strict convex order implies $W_\infty$-stability, which in turn implies irreducibility.
\begin{proposition}
    Assume $\mu,\nu \in \mathcal{P}(\R^d)$ are in convex order, non-degenerate, and, for any non-affine convex function $u:\textnormal{conv}(Y) \to \R$, it holds $\int u \dd \mu < \int u \dd \nu$. Then the pair $(\mu,\nu)$ is $W_\infty$-stable.
\end{proposition}

\begin{proof}
    Assume without loss of generality that $\int x \dd \mu = 0$. 
    Since $X \subset \ri ( \mathrm{conv}(Y) ) = \mathrm{int} (\mathrm{conv}(Y))$ and $X$ is compact, there is $\eta_1 > 0$ such that $X_{\eta_1} :=X + B(0,\eta_1) \subset \mathrm{int} (\mathrm{conv}(Y))$. Now set 
    \begin{equation*}
        S = \left\{u \in C(\mathrm{conv}(Y)): u \text{ convex}, u(0) = 0 =  \min_{x\in \mathrm{conv}(Y)} u(x), \mathrm{diam}(\partial u(X_\eta) ) \geq \frac 12,u \text{ 1-Lipschitz}\right\}
    \end{equation*}
    $S$ is relatively compact in $C(\mathrm{conv}(Y))$ by Arzelà--Ascoli and closed, hence compact.
    Thus
    \begin{equation*}
        \eta_2 :=\inf_{u \in S} \int u \dd (\nu-\mu) = \int u_0 \dd (\nu-\mu) >0
    \end{equation*}
    where $u_0 \in S$ is not affine and the last inequality is given by strict convex ordering. Set $\eta=\min\{\eta_1,\eta_2\}$. Let $\bar\mu\in \mathcal{P}(\R^d)$ be such that $\int x \dd \bar \mu = \int x \dd \mu$ and $W_\infty(\bar \mu,\mu) \leq \eta$. To check that $\bar \mu \leq_c \nu$ it is enough to do so for $1$-Lipschitz convex functions which satisfy $u(0) = 0 = \min_{x\in \mathrm{conv}(Y)} u(x)$. Let $u \in C(\mathrm{conv}(Y))$ be such a $1$-Lipschitz convex function. If $\delta := \mathrm{diam}(\partial u(X_\eta)) = 0$ then $u$ is affine on $X_\eta$; since $\mathrm{supp}(\bar \mu) \subset X_\eta$ and $\int x \dd \bar \mu = \int x \dd \mu$, $\int u\dd\bar\mu= \int u \dd\mu \leq \int u \dd \nu$. Else, set $\bar u = \frac{u}{2\delta} \square \Vert \cdot \Vert\in S$. We have
    \begin{equation*}
        \int \frac{u}{2\delta} \dd \bar\mu \leq \int \bar  u \dd \bar \mu \leq \int \bar u \dd \mu + \eta \leq \int \bar u \dd \nu \leq \int \frac{u}{2\delta} \dd \nu
    \end{equation*}
    where the first inequality holds from $\frac{u}{2\delta} = \bar u$, since $\partial u (X_\eta) \subset \overline B(0,1)$, on $X_\eta$ which contains the support of $\bar \mu$, the second by $W_\infty(\bar \mu,\mu) \leq \eta$ and $1$-Lipschitzness of $\bar u$, the third because $\bar u \in S$ and the last because the infimal convolution is smaller than the function itself. This concludes the proof of the $W_\infty$ stability claim.
\end{proof}

\begin{remark}
    Note that the non-degeneracy assumption is really needed. For instance, if $\mu = \frac12 \delta_0 + \frac12 \delta_1$ and $\nu = \frac34 \delta_0 + \frac14 \delta_2$, then $ m:= \frac12 \delta_{(0,0)}+ \frac14 \delta_{(1,0)}+\frac14 \delta_{(1,2)}$ is the unique martingale transport plan. As the conditional law $m_0 = \delta_0 \not\sim \nu$, $(\mu,\nu)$ is not irreducible. But for any convex function $u: [0,2] \to \R$,  $\int u \dd\mu = \int u \dd\nu $ yields $u(1)=\frac{u(0)+u(2)}{2}$, which in turn forces $u$ to be affine on $[0,2]$. This proves $(\mu,\nu)$ is in strict convex order in the sense of \cite{KaReWa26}, which in turn shows that strict convex order does not imply irreducibility in general.
\end{remark}

\bibliographystyle{abbrv}
\bibliography{references}

@misc{CaMaSy26,
  author       = {Carlier, Guillaume and Malamut, Hugo and Sylvestre, Maxime},
  title        = {Weak Optimal Transport with Moment Constraints:
                  Constraint Qualification, Dual Attainment and
                  Entropic Regularization},
  year         = {2026},
  eprint       = {2511.16211},
  archivePrefix = {arXiv},
  primaryClass = {math.OC}
}

@misc{GaNaTo25,
      title={Metric extrapolation in the Wasserstein space}, 
      author={Thomas O. Gallouët and Andrea Natale and Gabriele Todeschi},
      year={2025},
      eprint={2407.10516},
      archivePrefix={arXiv},
      primaryClass={math.OC},
      url={https://arxiv.org/abs/2407.10516} 
}

@misc{CheCoReWa26,
      title={Convergence of Sinkhorn's Algorithm for Entropic Martingale Optimal Transport Problem}, 
      author={Fan Chen and Giovanni Conforti and Zhenjie Ren and Xiaozhen Wang},
      year={2026},
      eprint={2407.14186},
      archivePrefix={arXiv},
      primaryClass={math.PR},
      url={https://arxiv.org/abs/2407.14186}, 
}

@misc{BaBeBiLe26,
      title={Bridging classical and martingale Schr\"odinger bridges}, 
      author={Julio Backhoff and Mathias Beiglböck and Giorgia Bifronte and Armand Ley},
      year={2026},
      eprint={2604.01299},
      archivePrefix={arXiv},
      primaryClass={math.PR},
      url={https://arxiv.org/abs/2604.01299}, 
}

@misc{KaWa26,
      title={Sinkhorn algorithms for entropic vector quantile regression}, 
      author={Kengo Kato and Boyu Wang},
      year={2026},
      eprint={2603.21554},
      archivePrefix={arXiv},
      primaryClass={math.OC},
      url={https://arxiv.org/abs/2603.21554}, 
}

@article{CaDuGaSu23,
  title   = {{SISTA}: Learning Optimal Transport Costs under Sparsity Constraints},
  author  = {Carlier, Guillaume and Dupuy, Arnaud and Galichon, Alfred and Sun, Yifei},
  journal = {Communications on Pure and Applied Mathematics},
  volume  = {76},
  number  = {9},
  pages   = {1659--1677},
  year    = {2023},
  doi     = {10.1002/cpa.22047}
}

@article{BaBeSchaTschi26,
  author  = {Backhoff-Veraguas, Julio and Beiglb{\"o}ck, Mathias and Schachermayer, Walter and Tschiderer, Bertram},
  title   = {Existence of Bass Martingales and the Martingale Benamou--Brenier Problem in {$\mathbb{R}^d$}},
  journal = {The Annals of Probability},
  year    = {2026},
  volume  = {54},
  number  = {5},
  pages   = {2216--2263}
}

@article {Stra65,
    AUTHOR = {Strassen, V.},
     TITLE = {The existence of probability measures with given marginals},
   JOURNAL = {Ann. Math. Statist.},
  FJOURNAL = {Annals of Mathematical Statistics},
    VOLUME = {36},
      YEAR = {1965},
     PAGES = {423--439},
      ISSN = {0003-4851},
   MRCLASS = {60.05 (60.20)},
  MRNUMBER = {177430},
MRREVIEWER = {J.\ Wolfowitz},
       DOI = {10.1214/aoms/1177700153},
       URL = {https://doi.org/10.1214/aoms/1177700153},
}

@article{CoHe21,
	author = {Conze, Antoine and Henry-Labord{\`e}re, Pierre},
	doi = {10.2139/ssrn.3853085},
	issn = {1556-5068},
	journal = {SSRN Electronic Journal},
	language = {en},
	shorttitle = {Bass {Construction} with {Multi}-{Marginals}},
	title = {Bass {Construction} with {Multi}-{Marginals}: {Lightspeed} {Computation} in a {New} {Local} {Volatility} {Model}},
	url = {https://www.ssrn.com/abstract=3853085},
	urldate = {2023-07-25},
	year = {2021}}

@article{AcMaPa25,
  author = {Acciaio, Beatrice and Marini, Antonio and Pammer, Gudmund},
  journal = {SIAM Journal on Financial Mathematics},
  title = {Calibration of the {Bass} Local Volatility Model},
  volume = {16},
  number = {3},
  pages = {803--833},
  year = {2025},
  doi = {10.1137/23M1622660},
}

@article{BeJu16,
  author = {Beiglb{\"o}ck, M. and Juillet, N.},
  doi = {10.1214/14-AOP966},
  fjournal = {The Annals of Probability},
  issn = {0091-1798},
  journal = {Ann. Probab.},
  mrclass = {60G46 (49K45 49N05 49Q20)},
  mrnumber = {3456332},
  number = {1},
  pages = {42--106},
  title = {On a problem of optimal transport under marginal martingale constraints},
  url = {http://dx.doi.org/10.1214/14-AOP966},
  volume = {44},
  year = {2016},
}

@article{BaBeHuKa20,
  author = {Backhoff-Veraguas, Julio and Beiglb\"{o}ck, Mathias and Huesmann, Martin and K{\"a}llblad, Sigrid},
  journal = {Annals of Probability},
  title = {Martingale {Benamou--Brenier}: A probabilistic perspective},
  volume = {48},
  number = {5},
  pages = {2258--2289},
  year = {2020},
  doi = {10.1214/20-AOP1422},
}

@article{HuTr17,
  author = {Huesmann, Martin and Trevisan, Dario},
  doi = {10.3150/18-BEJ1069},
  fjournal = {Bernoulli. Official Journal of the Bernoulli Society for Mathematical Statistics and Probability},
  issn = {1350-7265},
  journal = {Bernoulli},
  mrclass = {60G46 (49Q20 60B10)},
  mrnumber = {4003563},
  number = {4A},
  pages = {2729--2757},
  title = {A {B}enamou-{B}renier formulation of martingale optimal transport},
  url = {https://doi.org/10.3150/18-BEJ1069},
  volume = {25},
  year = {2019},
}

@misc{NuWi24,
  author = {Nutz, Marcel and Wiesel, Johannes},
  title = {On the Martingale {Schr{\"o}dinger} Bridge between Two Distributions},
  year = {2024},
  howpublished = {Preprint, {arXiv}:2401.05209},
  url = {https://arxiv.org/abs/2401.05209},
  arxiv = {arXiv:2401.05209},
}

@misc{KaReWa26,
      title={Exponential convergence of Sinkhorn algorithm for entropy martingale optimal transport}, 
      author={Anna Kazeykina and Zhenjie Ren and Hecheng Wang},
      year={2026},
      eprint={2609.18850},
      archivePrefix={arXiv},
      primaryClass={math.PR},
      url={https://arxiv.org/abs/2609.18850}, 
}

@inproceedings{Cu13,
  author = {Cuturi, Marco},
  booktitle = {Proceedings of the 26th International Conference on Neural Information Processing Systems},
  pages = {2292--2300},
  title = {{Sinkhorn Distances: Lightspeed Computation of Optimal Transport}},
  volume = {2},
  year = {2013},
}

@article{GoRoSaShTe18,
  author  = {Gozlan, Nathael and Roberto, Cyril and Samson, Paul-Marie and Shu, Yan and Tetali, Prasad},
  title   = {Characterization of a Class of Weak Transport-Entropy Inequalities on the Line},
  journal = {Annales de l'Institut Henri Poincar{\'e}, Probabilit{\'e}s et Statistiques},
  volume  = {54},
  number  = {3},
  pages   = {1667--1693},
  year    = {2018}
}

@article{GoRoSaTe17,
  author  = {Gozlan, Nathael and Roberto, Cyril and Samson, Paul-Marie and Tetali, Prasad},
  title   = {Kantorovich Duality for General Transport Costs and Applications},
  journal = {Journal of Functional Analysis},
  volume  = {273},
  number  = {11},
  pages   = {3327--3405},
  year    = {2017}
}

@article{AlBoCh19,
  author  = {Alibert, Jean-Jacques and Bouchitt{\'e}, Guy and Champion, Thierry},
  title   = {A New Class of Costs for Optimal Transport Planning},
  journal = {European Journal of Applied Mathematics},
  volume  = {30},
  number  = {6},
  pages   = {1229--1263},
  year    = {2019}
}

@article{BaBePa19,
  author  = {Backhoff-Veraguas, Julio and Beiglb{\"o}ck, Mathias and Pammer, Gudmund},
  title   = {Existence, Duality, and Cyclical Monotonicity for Weak Transport Costs},
  journal = {Calculus of Variations and Partial Differential Equations},
  volume  = {58},
  number  = {6},
  pages   = {1--28},
  year    = {2019}
}

@article{BaPa22,
  author  = {Backhoff-Veraguas, Julio and Pammer, Gudmund},
  title   = {Applications of Weak Transport Theory},
  journal = {Bernoulli},
  volume  = {28},
  number  = {1},
  pages   = {370--394},
  year    = {2022}
}

@misc{BePaRiSc25,
  author        = {Beiglb{\"o}ck, Mathias and Pammer, Gudmund and Riess, Leonard and Schrott, Stefan},
  title         = {The Fundamental Theorem of Weak Optimal Transport},
  year          = {2025},
  eprint        = {2501.16316},
  archivePrefix = {arXiv}
}

@article{BeLu13,
  title = {On the Convergence of Block Coordinate Descent Type Methods},
  volume = {23},
  ISSN = {1095-7189},
  url = {http://dx.doi.org/10.1137/120887679},
  DOI = {10.1137/120887679},
  number = {4},
  journal = {SIAM Journal on Optimization},
  publisher = {Society for Industrial & Applied Mathematics (SIAM)},
  author = {Beck,  Amir and Tetruashvili,  Luba},
  year = {2013},
  month = Jan,
  pages = {2037–2060}
}

@article{ChaDePe08,
  title = {The $\infty$-Wasserstein Distance: Local Solutions and Existence of Optimal Transport Maps},
  volume = {40},
  ISSN = {1095-7154},
  url = {http://dx.doi.org/10.1137/07069938X},
  DOI = {10.1137/07069938x},
  number = {1},
  journal = {SIAM Journal on Mathematical Analysis},
  publisher = {Society for Industrial & Applied Mathematics (SIAM)},
  author = {Champion,  Thierry and De Pascale,  Luigi and Juutinen,  Petri},
  year = {2008},
  month = Jan,
  pages = {1–20}
}

@article{XuYin13,
  title = {A Block Coordinate Descent Method for Regularized Multiconvex Optimization with Applications to Nonnegative Tensor Factorization and Completion},
  volume = {6},
  ISSN = {1936-4954},
  url = {http://dx.doi.org/10.1137/120887795},
  DOI = {10.1137/120887795},
  number = {3},
  journal = {SIAM Journal on Imaging Sciences},
  publisher = {Society for Industrial & Applied Mathematics (SIAM)},
  author = {Xu,  Yangyang and Yin,  Wotao},
  year = {2013},
  month = Jan,
  pages = {1758–1789}
}

@article{RaHoLu13,
  title = {A Unified Convergence Analysis of Block Successive Minimization Methods for Nonsmooth Optimization},
  volume = {23},
  ISSN = {1095-7189},
  url = {http://dx.doi.org/10.1137/120891009},
  DOI = {10.1137/120891009},
  number = {2},
  journal = {SIAM Journal on Optimization},
  publisher = {Society for Industrial & Applied Mathematics (SIAM)},
  author = {Razaviyayn,  Meisam and Hong,  Mingyi and Luo,  Zhi-Quan},
  year = {2013},
  month = Jan,
  pages = {1126–1153}
}

@inproceedings{AlBaCaPh25,
  author    = {Alouadi, Alexandre and Barreau, Baptiste
               and Carlier, Laurent and Pham, Huy{\^e}n},
  title     = {Robust Time Series Generation via
               {Schr\"odinger} Bridge: A Comprehensive Evaluation},
  booktitle = {Proceedings of the 6th ACM International
               Conference on AI in Finance},
  year      = {2025},
  pages     = {906--914},
  publisher = {ACM},
  doi       = {10.1145/3768292.3770391},
  url       = {https://arxiv.org/abs/2503.02943}
}

@misc{AlHeLoMaPhTo26,
  author        = {Alouadi, Alexandre and
                   Henry-Labord{\`e}re, Pierre and
                   Loeper, Gr{\'e}goire and
                   Mazhar, Othmane and
                   Pham, Huy{\^e}n and Touzi, Nizar},
  title         = {{LightSBB-M}: Bridging {Schr\"odinger}
                   and Bass for Generative Diffusion Modeling},
  year          = {2026},
  eprint        = {2601.19312},
  archivePrefix = {arXiv},
  primaryClass  = {cs.LG},
  url           = {https://arxiv.org/abs/2601.19312}
}

@misc{DePhZa26,
  author        = {De Marco, Stefano and
                   Pham, Huy{\^e}n and Zanni, Davide},
  title         = {{Schr\"odinger} Bridges with Jumps
                   for Time Series Generation},
  year          = {2026},
  eprint        = {2602.20011},
  archivePrefix = {arXiv},
  primaryClass  = {q-fin.MF},
  url           = {https://arxiv.org/abs/2602.20011}
}

@article{CaFeMe64,
  author  = {Cartier, P. and Fell, J. M. G. and Meyer, P.-A.},
  title   = {Comparaison des mesures port\'ees par un ensemble convexe compact},
  journal = {Bull. Soc. Math. France},
  volume  = {92},
  pages   = {435--445},
  year    = {1964}
}

@article{Ho98,
  author  = {Hobson, D. G.},
  title   = {Robust hedging of the lookback option},
  journal = {Finance Stoch.},
  volume  = {2},
  number  = {4},
  pages   = {329--347},
  year    = {1998}
}

@incollection{Ho11,
  author    = {Hobson, D.},
  title     = {The {S}korokhod embedding problem and model-independent bounds for option prices},
  booktitle = {Paris--{P}rinceton Lectures on Mathematical Finance 2010},
  series    = {Lecture Notes in Math.},
  volume    = {2003},
  pages     = {267--318},
  publisher = {Springer},
  year      = {2011}
}

@book{He17,
  author    = {Henry-Labord{\`e}re, P.},
  title     = {Model-free Hedging: A Martingale Optimal Transport Viewpoint},
  publisher = {Chapman and Hall/CRC},
  year      = {2017}
}

@article{BeHePe13,
  author  = {Beiglb{\"o}ck, M. and Henry-Labord{\`e}re, P. and Penkner, F.},
  title   = {Model-independent bounds for option prices -- a mass transport approach},
  journal = {Finance Stoch.},
  volume  = {17},
  number  = {3},
  pages   = {477--501},
  year    = {2013}
}

@article{GaHeTo14,
  author  = {Galichon, A. and Henry-Labord{\`e}re, P. and Touzi, N.},
  title   = {A stochastic control approach to no-arbitrage bounds given marginals, with an application to lookback options},
  journal = {Ann. Appl. Probab.},
  volume  = {24},
  number  = {1},
  pages   = {312--336},
  year    = {2014}
}

@article{DoSo14,
  author  = {Dolinsky, Y. and Soner, H. M.},
  title   = {Martingale optimal transport and robust hedging in continuous time},
  journal = {Probab. Theory Related Fields},
  volume  = {160},
  number  = {1-2},
  pages   = {391--427},
  year    = {2014}
}

@article{Ob04,
  author  = {Ob{\l}{\'o}j, J.},
  title   = {The {S}korokhod embedding problem and its offspring},
  journal = {Probab. Surv.},
  volume  = {1},
  pages   = {321--392},
  year    = {2004}
}

@article{BeCoHu17,
  author  = {Beiglb{\"o}ck, M. and Cox, A. M. G. and Huesmann, M.},
  title   = {Optimal transport and {S}korokhod embedding},
  journal = {Invent. Math.},
  volume  = {208},
  number  = {2},
  pages   = {327--400},
  year    = {2017}
}

@article{HeTo16,
  author  = {Henry-Labord{\`e}re, P. and Touzi, N.},
  title   = {An explicit martingale version of the one-dimensional {B}renier theorem},
  journal = {Finance Stoch.},
  volume  = {20},
  number  = {3},
  pages   = {635--668},
  year    = {2016}
}

@article{BeHeTo17,
  author  = {Beiglb{\"o}ck, M. and Henry-Labord{\`e}re, P. and Touzi, N.},
  title   = {Monotone martingale transport plans and {S}korokhod embedding},
  journal = {Stochastic Process. Appl.},
  volume  = {127},
  number  = {9},
  pages   = {3005--3013},
  year    = {2017}
}

@article{HoKl15,
  author  = {Hobson, D. and Klimmek, M.},
  title   = {Robust price bounds for the forward starting straddle},
  journal = {Finance Stoch.},
  volume  = {19},
  number  = {1},
  pages   = {189--214},
  year    = {2015}
}

@article{GhKiLi19,
  author  = {Ghoussoub, N. and Kim, Y.-H. and Lim, T.},
  title   = {Structure of optimal martingale transport plans in general dimensions},
  journal = {Ann. Probab.},
  volume  = {47},
  number  = {1},
  pages   = {109--164},
  year    = {2019}
}

@article{BeJu21,
  author  = {Beiglb{\"o}ck, M. and Juillet, N.},
  title   = {Shadow couplings},
  journal = {Trans. Amer. Math. Soc.},
  volume  = {374},
  number  = {7},
  pages   = {4973--5002},
  year    = {2021}
}

@article{BaLoOb25,
  author  = {Backhoff, J. and Loeper, G. and Ob{\l}{\'o}j, J.},
  title   = {Geometric martingale {B}enamou--{B}renier transport and geometric {B}ass martingales},
  journal = {Proc. Amer. Math. Soc.},
  volume  = {153},
  number  = {11},
  pages   = {4945--4960},
  year    = {2025}
}

@article{GuOb19,
  author  = {Guo, G. and Ob{\l}{\'o}j, J.},
  title   = {Computational methods for martingale optimal transport problems},
  journal = {Ann. Appl. Probab.},
  volume  = {29},
  number  = {6},
  pages   = {3311--3347},
  year    = {2019}
}

@article{AlCoJo19,
  author  = {Alfonsi, A. and Corbetta, J. and Jourdain, B.},
  title   = {Sampling of one-dimensional probability measures in the convex order and computation of robust option price bounds},
  journal = {Int. J. Theor. Appl. Finance},
  volume  = {22},
  number  = {3},
  pages   = {1950002},
  year    = {2019}
}

@article{EcGuLiOb21,
  author  = {Eckstein, S. and Guo, G. and Lim, T. and Ob{\l}{\'o}j, J.},
  title   = {Robust pricing and hedging of options on multiple assets and its numerics},
  journal = {SIAM J. Financial Math.},
  volume  = {12},
  number  = {1},
  pages   = {158--188},
  year    = {2021}
}

@article{BoTa13,
  author  = {Bonnans, J. F. and Tan, X.},
  title   = {A model-free no-arbitrage price bound for variance options},
  journal = {Appl. Math. Optim.},
  volume  = {68},
  pages   = {43--73},
  year    = {2013}
}

@article{TaTo13,
  author  = {Tan, X. and Touzi, N.},
  title   = {Optimal transportation under controlled stochastic dynamics},
  journal = {Ann. Probab.},
  volume  = {41},
  number  = {5},
  pages   = {3201--3240},
  year    = {2013}
}

@article{GuLo21,
  author  = {Guo, I. and Loeper, G.},
  title   = {Path dependent optimal transport and model calibration on exotic derivatives},
  journal = {Ann. Appl. Probab.},
  volume  = {31},
  number  = {3},
  pages   = {1232--1263},
  year    = {2021}
}

@article{GuLoWa22,
  author  = {Guo, I. and Loeper, G. and Wang, S.},
  title   = {Calibration of local-stochastic volatility models by optimal transport},
  journal = {Math. Finance},
  volume  = {32},
  number  = {1},
  pages   = {46--77},
  year    = {2022}
}

@misc{HJLOP,
  author = {Hasenbichler, M. and Joseph, B. and Loeper, G. and Ob{\l}{\'o}j, J. and Pammer, G.},
  title  = {The martingale {S}inkhorn algorithm},
  note   = {Preprint, arXiv:2310.13797},
  year   = {2023}
}

@article{BaPaSc25,
  author  = {Backhoff-Veraguas, J. and Pammer, G. and Schachermayer, W.},
  title   = {The {$L^2$} gradient flow of the {B}ass functional in martingale optimal transport},
  journal = {Math. Financ. Econ.},
  volume  = {19},
  number  = {4},
  pages   = {665--686},
  year    = {2025}
}

@article{GuJuTa26,
  author  = {Guo, G. and Juillet, N. and Tang, W.},
  title   = {The football model, stochastic ordering and martingale transport},
  journal = {Bernoulli},
  volume  = {32},
  number  = {3},
  pages   = {2323--2350},
  year    = {2026},
  doi     = {10.3150/25-BEJ1955},
  note    = {arXiv:2503.07145}
}

@misc{HuSoTu26,
  author = {Hua, D. and Song, A. and Tudose, S.},
  title  = {On {T}alagrand's convexity conjecture},
  note   = {Preprint, arXiv:2605.10908},
  year   = {2026}
}

@misc{De18,
  author = {De March, H.},
  title  = {Entropic approximation for multi-dimensional martingale optimal transport},
  note   = {Preprint, arXiv:1812.11104},
  year   = {2018}
}

@article{EcLa26,
  author  = {Eckstein, S. and Lakhal, A.},
  title   = {Exponential convergence of general iterative proportional fitting procedures},
  journal = {SIAM J. Optim.},
  volume  = {36},
  number  = {2},
  pages   = {912--937},
  year    = {2026}
}

\end{document}